\documentclass[times,sort&compress,3p]{elsarticle}
\journal{}

\RequirePackage{amsthm,amsmath,amsfonts,amssymb}
\RequirePackage[colorlinks,citecolor=blue,urlcolor=blue]{hyperref}
\RequirePackage{graphicx}
\RequirePackage{bm}
\RequirePackage{color}
\RequirePackage{calrsfs}
\RequirePackage{enumitem}
\setitemize{noitemsep,topsep=0pt,parsep=0pt,partopsep=0pt}   

\makeatletter
\def\ps@pprintTitle{%
 \let\@oddhead\@empty
 \let\@evenhead\@empty
 \def\@oddfoot{}%
 \let\@evenfoot\@oddfoot}
\makeatother

\usepackage[labelfont=bf]{caption}

\theoremstyle{plain}
\newtheorem{prop}{Proposition}

\newtheorem{thm}[prop]{Theorem}
\newtheorem{cor}[prop]{Corollary}
\newtheorem{lem}[prop]{Lemma}
\newtheorem{cond}[prop]{Condition}

\newtheoremstyle{remark}
  {}{}{}{}{\bfseries}{.}{.5em}{{\thmname{#1 }}{\thmnumber{#2}}{\thmnote{ (#3)}}}
\theoremstyle{remark}
\newtheorem{remark}[prop]{Remark}

\newcommand{\eps}{\varepsilon}
\newcommand{\N}{\mathbb{N}}
\newcommand{\Z}{\mathbb{Z}}
\newcommand{\R}{\mathbb{R}}

\newcommand{\Bb}{\mathbb{B}}
\newcommand{\Cb}{\mathbb{C}}

\newcommand{\Xb}{\mathbb{X}}
\newcommand{\Zb}{\mathbb{Z}}

\newcommand{\Bc}{\mathcal{B}}
\newcommand{\Cc}{\mathcal{C}}
\newcommand{\Fc}{\mathcal{F}}
\newcommand{\Kc}{\mathcal{K}}
\newcommand{\Gc}{\mathcal{G}}

\newcommand{\Xc}{\mathcal{X}}
\newcommand{\Ac}{\mathcal{A}}
\newcommand{\Wc}{\mathcal{W}}

\newcommand{\Beta}{\mathrm{beta}}
\newcommand{\dd}{\mathrm{d}}

\newcommand{\Ex}{\mathbb{E}}
\newcommand{\Var}{\mathrm{Var}}
\newcommand{\Cov}{\mathrm{Cov}}
\newcommand{\1}{\mathbf{1}}
\newcommand{\ip}[1]{\lfloor #1 \rfloor}
\newcommand{\up}[1]{\lceil #1 \rceil}
\renewcommand{\Pr}{\mathbb{P}}
\newcommand{\pobs}[1]{\hat{\bm #1}}
\newcommand{\leqst}{\leq_{\text{st}}}
\newcommand{\leqlr}{\leq_{\text{lr}}}

\newcommand{\sss}[1]{\scriptscriptstyle{#1}}

\newcommand{\ik}[1]{{\color{black} {#1}}}

\allowdisplaybreaks

\begin{document}

\begin{frontmatter}
  
\title{A class of smooth, possibly data-adaptive nonparametric copula estimators containing the empirical beta copula}

\author[A1]{Ivan Kojadinovic\corref{mycorrespondingauthor}}
\author[A1,A2]{Bingqing Yi}

\address[A1]{CNRS / Universit\'e de Pau et des Pays de l'Adour / E2S UPPA, Laboratoire de math\'ematiques et applications -- IPRA, UMR 5142, B.P. 1155, 64013 Pau Cedex, France.}
\address[A2]{School of Mathematics \& Statistics, The University of Melbourne, Parkville, VIC 3010, Australia.}

\cortext[mycorrespondingauthor]{Corresponding author. Email address: \url{ivan.kojadinovic@univ-pau.fr}}

\begin{abstract}
  A broad class of smooth, possibly data-adaptive nonparametric copula estimators that contains empirical Bernstein copulas introduced by Sancetta and Satchell (and thus the empirical beta copula proposed by Segers, Sibuya and Tsukahara) is studied. Within this class, a subclass of estimators that depend on a scalar parameter determining the amount of marginal smoothing and a functional parameter controlling the shape of the smoothing region is specifically considered. Empirical investigations of the influence of these parameters suggest to focus on two particular data-adaptive smooth copula estimators that were found to be uniformly better than the empirical beta copula in all of the considered Monte Carlo experiments. Finally, with future applications to change-point detection in mind, conditions under which related sequential empirical copula processes converge weakly are provided. 
\end{abstract}


\end{frontmatter}


\section{Introduction}


Let $\Xc_{1:n} = (\bm X_1,\dots,\bm X_n)$ be a stretch of $d$-dimensional random vectors from a stationary time series $(\bm X_i)_{i \in \Z}$. The distribution function (d.f.) $F$ of $\bm X_1$ is assumed to have continuous univariate margins $F_1,\dots,F_d$. As a consequence of a well-known theorem of \citet{Skl59}, the multivariate d.f.\ $F$ can be expressed as
\begin{equation*}
F(\bm x) = C\{F_1(x_1),\dots,F_d(x_d)\}, \qquad \bm x \in \R^d,
\end{equation*}
in terms of a unique copula $C$, that is, a unique $d$-dimensional d.f.\ with standard uniform margins.

To carry out statistical inference on the unknown copula $C$ using the available observations $\Xc_{1:n}$, it is often necessary to have at hand nonparametric estimators of $C$. The best known such estimator is called the empirical copula \cite{Deh79}. Under a rather weak condition (see Condition~\ref{cond:Bn} in Section~\ref{sec:asym}), the latter is asymptotically equivalent to the empirical d.f.\ of the multivariate ranks obtained from $\Xc_{1:n}$ scaled by $1/n$ which was studied in \cite{Rus76}. Two smooth versions that are genuine copulas when there are no ties in the components samples of $\Xc_{1:n}$ are the empirical checkerboard copula \cite[see, e.g.,][and the references therein]{GenNesRem17,GenNes07} and the empirical beta copula proposed in~\cite{SegSibTsu17}. The latter was found to have better small-sample properties than the former and the classical empirical copula in the Monte Carlo experiments reported in \cite{SegSibTsu17}.

In this work, we investigate extensions of the construction that allowed \citet{SegSibTsu17} to study the asymptotics of empirical Bernstein copulas introduced in \cite{SanSat04}, and thus of the empirical beta copula. The initial motivation for this undertaking stems from an early attempt to obtain sequential versions of the asymptotic results of \cite{SegSibTsu17} (with an application to change-point detection in mind) during which it appeared that alternative ways of smoothing could be considered. In particular, we allow the underlying smoothing distributions to depend on the data, leading to a rather broad class of smooth, possibly data-adaptive empirical copulas.

The paper is organized as follows. In Section~\ref{sec:class}, we define a broad class of smooth, possibly data-adaptive nonparametric copula estimators that contains empirical Bernstein copulas, and thus the empirical beta copula. Conditions under which such smooth estimators have standard uniform univariate margins, are multivariate d.f.s or are genuine copulas are provided in Section~\ref{sec:properties}. In Section~\ref{sec:specific}, we focus on a subclass of empirical copulas that depend on a scalar parameter that determines the amount of marginal smoothing and a functional parameter that controls the shape of the smoothing region in $[0,1]^d$. Using an implementation for the \textsf{R} statistical environment \cite{Rsystem} (available on the web page of the first author), we investigate the influence of these parameters through Monte Carlo experiments in Section~\ref{sec:MC:est} and, as a result, we suggest to focus on two specific smooth data-adaptive copula estimators that were found to be uniformly better than the empirical beta copula in all of the considered Monte Carlo experiments. Finally, in Section~\ref{sec:asym}, we study the weak convergence of the sequential empirical copula processes related to the general class of smooth estimators proposed in Section~\ref{sec:class} and, in particular, establish conditions under which they are asymptotically equivalent to the classical sequential empirical copula process initially studied in \cite{BucKoj16}.


\section{A broad class of smooth, possibly data-adaptive empirical copulas}
\label{sec:class}

For any $j \in \{1,\dots,d\}$, let $F_{1:n,j}$ be the empirical d.f.\ computed from the $j$th component sample $X_{1j},\dots,X_{nj}$ of the available observations $\Xc_{1:n}$. Then, let
$$
R_{ij}^{1:n} = n F_{1:n,j}(X_{ij}) = \sum_{t=1}^n \1( X_{tj} \le X_{ij} )
$$
be the (maximal) rank of $X_{ij}$ among $X_{1j},\dots,X_{nj}$. Furthermore, let $\bm R^{1:n}_i =  \left(R_{i1}^{1:n}, \dots, R_{id}^{1:n} \right)$ and $\pobs{U}^{1:n}_i = \bm R^{1:n}_i / n$, $i \in \{1,\dots,n\}$, be the multivariate ranks and the multivariate scaled ranks, respectively, obtained from $\Xc_{1:n}$. Note that the $d$-dimensional random vectors $\pobs{U}^{1:n}_1, \dots, \pobs{U}^{1:n}_n$ are sometimes referred to as pseudo-observations from $C$ \cite[see, e.g.,][Chapter~4]{HofKojMaeYan18}. Following \cite{Rus76}, the empirical copula $C_{1:n}$ of $\Xc_{1:n}$ at $\bm u = (u_1,\dots,u_d) \in [0,1]^d$ can then be defined by
\begin{equation}
\label{eq:C:1n}
C_{1:n}(\bm u) = \frac{1}{n} \sum_{i=1}^n  \prod_{j=1}^d \1\left( R_{ij}^{1:n} / n \leq u_j \right)
= \frac{1}{n} \sum_{i=1}^n  \1(\pobs{U}^{1:n}_i \leq \bm u),
\end{equation}
where inequalities between vectors are to be understood componentwise.

A smooth version of the empirical copula $C_{1:n}$ was proposed in \cite{SegSibTsu17} by replacing indicator functions in~\eqref{eq:C:1n} by d.f.s of particular beta distributions. Specifically, the empirical beta copula of $\Xc_{1:n}$ is defined by
\begin{equation}
\label{eq:C:1n:beta}
  C_{1:n}^\Beta(\bm{u}) = \frac{1}{n} \sum_{i=1}^n \prod_{j=1}^d \Beta_{n,R_{ij}^{1:n}}(u_j), \qquad \bm u = (u_1,\dots,u_d) \in [0,1]^d,
\end{equation}
where, for any $n \in \N$ and $r \in \{1,\dots,n\}$, $\Beta_{n,r}$ denotes the d.f.\ of the distribution Beta$(r, n+1-r)$ (the beta distribution with shape parameters $\alpha = r$ and $\beta = n+1-r$). When there are no ties in the component samples of $\Xc_{1:n}$, Lemma~2.6 in \cite{SegSibTsu17} states that the empirical beta copula is actually a particular case of the empirical Bernstein copula introduced in \cite{SanSat04} and further studied in \cite{JanSwaVer12}. Proposition~2.8 in \cite{SegSibTsu17} additionally shows that the supremum distance between the empirical beta copula $C_{1:n}^\Beta$ and the classical empirical copula $C_{1:n}$ is $O(n^{-1/2} (\ln n)^{1/2})$, thereby suggesting that the empirical beta copula is a smoothing of $C_{1:n}$ at approximately bandwidth $O(n^{-1/2})$; see also Corollary 3.7 in \cite{SegSibTsu17}.

To study the asymptotics of the empirical beta copula, \citet{SegSibTsu17} cleverly used the fact that it could be written as a mixture involving the classical empirical copula. For any $n \in \N$ and $\bm u \in [0,1]^d$, let $\bm \mu_{n,\bm u}$ be the law of the $d$-dimensional random vector $(S_{n,1,u_1}/n,\dots S_{n,d,u_d}/n)$, where $S_{n,1,u_1},\dots S_{n,d,u_d}$ are independent random variables and, for each $j \in \{1,\dots d\}$, $S_{n,j,u_j}$ is Binomial$(n,u_j)$, that is, $S_{n,j,u_j}$ follows a binomial distribution with parameters $n$ and $u_j$. The following lemma proven in Appendix~\ref{proofs:short} is instrumental for understanding the approach of \cite{SegSibTsu17}.

\begin{lem}
  \label{lem:beta:mixture}
For any $\bm u \in [0,1]^d$, $n \in \N$ and $r_1,\dots,r_d \in \{1,\dots,n\}$,
$$
\prod_{j=1}^d \Beta_{n,r_j}(u_j) = \int_{[0,1]^d} \prod_{j=1}^d \1(r_j/n \leq w_j) \dd \mu_{n,\bm u}(\bm w).
$$
\end{lem}

Using Lemma~\ref{lem:beta:mixture}, by linearity of the integral, one immediately obtains that, for any $\bm u \in [0,1]^d$,
\begin{equation}
  \label{eq:C:1n:beta:mix}
  C_{1:n}^\Beta(\bm{u}) = \frac{1}{n} \sum_{i=1}^n  \int_{[0,1]^d} \prod_{j=1}^d \1 \left( R_{ij}^{1:n} / n \leq w_j \right) \dd \mu_{n,\bm u}(\bm w) = \int_{[0,1]^d} C_{1:n}(\bm w) \dd \mu_{n,\bm u}(\bm w).
\end{equation}

The main aim of this work is to study generalizations of the empirical beta copula based on alternative smoothing distributions, possibly depending on the data. As we continue, for any $m \in \N$, $\bm x \in (\R^d)^m$ and $\bm u \in [0,1]^d$, $\nu_{\bm u}^{\bm x}$ will denote a law on $[0,1]^d$ such that, for any $j \in \{1,\dots,d\}$, $\int_{[0,1]^d} w_j \dd \nu_{\bm u}^{\bm x}(\bm w) = u_j$; it is meant to be a generalization of $\mu_{n,\bm u}$ in~\eqref{eq:C:1n:beta:mix} that possibly depends on the data set $\bm x$. Furthermore, let $p \geq d$ be a fixed integer, let $\bm U$ be a $p$-dimensional random vector whose components are independent and standard uniform, and consider the following assumption.

\begin{cond}[Construction of smoothing random vectors]
  \label{cond:construction}
  For any $m \in \N$, $\bm x \in (\R^d)^m$ and $\bm u \in [0,1]^d$, there exists a function $\Wc_{\bm u}^{\bm x}: [0,1]^p \to [0,1]^d$ such that $\bm W_{\bm u}^{\bm x} = \Wc_{\bm u}^{\bm x}(\bm U)$ is a $[0,1]^d$-valued random vector with law $\nu_{\bm u}^{\bm x}$.
\end{cond}

To be able to define, for any $n \in \N$, $\Xc_{1:n} = (\bm X_1,\dots, \bm X_n)$ and, for any $m \leq n$, the random vectors $\bm W_{\bm u}^{\bm x}$,  $\bm x \in (\R^d)^m$, $\bm u \in [0,1]^d$, on the same probability space (the case $m < n$ will be needed in Section~\ref{sec:asym}), we assume that the underlying probability space $(\Omega, \Ac, \Pr)$ has a product structure, that is, $\Omega=\Omega_0 \times \Omega_1$ with probability measure $\Pr=\Pr_0 \otimes \Pr_1$, where $\Pr_i$ denotes the probability measure on~$\Omega_i$, such that, for any $\omega \in \Omega$, $\bm \Xc_{1:n}(\omega)$ only depends on the first coordinate of $\omega$ and $\bm U(\omega)$ only depends on the second coordinate of~$\omega$, implying in particular that $\Xc_{1:n}$ and $\bm U$ are independent. In that case, it can be verified using Fubini's theorem that $(\bm x, A) \mapsto  \Pr_1 \{ \Wc_{\bm u}^{\bm x}(\bm U) \in A \} = \Pr_1( \bm W_{\bm u}^{\bm x} \in A) = \nu_{\bm u}^{\bm x}(A)$ defines a regular version of the conditional distribution $\Pr(\bm W_{\bm u}^{\bm \Xc_{1:n}} \in \cdot \mid \Xc_{1:n})$ of $\bm W_{\bm u}^{\bm \Xc_{1:n}}$ given $\Xc_{1:n}$. As a consequence, in the rest of the paper, for an arbitrary real-valued function $h$, $\Ex \{ h(\bm W_{\bm u}^{\bm \Xc_{1:n}})\mid \bm \Xc_{1:n} \}$ is to be understood as $\int_{[0,1]^d} h(\bm w) \dd \nu_{\bm u}^{\sss \Xc_{1:n}}(\bm w)$.

As we continue, for any $m \in \N$, $\bm x \in (\R^d)^m$ and $\bm u \in [0,1]^d$, the $d$ components of the random vector $\bm W_{\bm u}^{\bm x}$ will be denoted by $W_{1,u_1}^{\bm x}, \dots, W_{d,u_d}^{\bm x}$ to indicate that the $j$th component of $\bm W_{\bm u}^{\bm x}$ depends on $u_j$ but not on $u_1,\dots,u_{j-1},u_{j+1},\dots,u_d$. It is however important to keep in mind that the joint distribution of $W_{1,u_1}^{\bm x}, \dots, W_{d,u_d}^{\bm x}$ may still depend on~$\bm u$. The fact that,  for any $\bm x \in (\R^d)^m$ and $\bm u \in [0,1]^d$, $\bm W_{\bm u}^{\bm x}$ has its support included in $[0,1]^d$ with $\Ex(\bm W_{\bm u}^{\bm x}) = \bm u$ implies that, for any $j \in \{1,\dots,d\}$, $\Var(W_{j, u_j}^{\bm x}) = 0$ if $u_j \in \{0,1\}$. Note that, more generally, for any $\bm x \in (\R^d)^m$, $\bm u \in [0,1]^d$ and $j \in \{1,\dots,d\}$, one has that 
\begin{equation}
  \label{eq:var:W}
  \begin{split}
    \Var(W_{j,u_j}^{\bm x}) &= \Ex \left\{ (W_{j,u_j}^{\bm x})^2 \right\}  - u_j^2 \leq   \Ex(W_{j,u_j}^{\bm x}) - u_j^2 = u_j (1 - u_j).
  \end{split}
\end{equation}

By analogy with~\eqref{eq:C:1n:beta:mix}, we then define alternative smooth versions of the empirical copula $C_{1:n}$ of the sample $\Xc_{1:n} = (\bm X_1,\dots, \bm X_n)$ in~\eqref{eq:C:1n} by
\begin{equation}
  \label{eq:C:1n:nu}
  C_{1:n}^\nu(\bm u) = \int_{[0,1]^d} C_{1:n}(\bm w) \dd \nu_{\bm u}^{\sss \Xc_{1:n}}(\bm w), 
  \qquad \bm u \in [0,1]^d.
\end{equation}
Roughly speaking, for any $\bm u \in [0,1]^d$, $C_{1:n}^\nu(\bm u)$ can be thought of as a ``weighted average'' of $C_{1:n}(\bm w)$ for $\bm w$ ``in a neighborhood of $\bm u$'' according to the smoothing distribution $\nu_{\bm u}^{\sss \Xc_{1:n}}$ (that may depend on the available observations $\Xc_{1:n}$).

\begin{remark}
  Let us comment on Condition~\ref{cond:construction}. For $n \in \N$, $\bm x \in (\R^d)^n$ and $\bm u \in [0,1]^d$, let $H$ be the d.f.\ corresponding to $\nu_{\bm u}^{\bm x}$, let $D$ be a copula of $H$ and let $H_1,\dots,H_d$ be the $d$ univariate margins of $H$. Assume for instance first that $D$ is an absolutely continuous copula. Then, using the inverse of the well-known transformation of \citet{Ros52}, it is possible to obtain, from the first $d$ components of $\bm U$, a $d$-dimensional random vector $\bm V$ with d.f.\ $D$ \citep[see, e.g.,][Section 2.7 for more details]{HofKojMaeYan18}. The random vector $\bm W_{\bm u}^{\bm x}$ can then be defined by $(H_1^{-1}(V_1), \dots, H_d^{-1}(V_d))$, where $H_1^{-1},\dots,H_d^{-1}$ are the quantile functions (generalized inverses) obtained from the $d$ univariate margins of $H$. As another example, assume that $D$ is the empirical beta copula obtained from the data set $\bm x \in (\R^d)^n$ whose component samples contain no ties. Let $r_{ij}$, $i \in \{1,\dots,n\}$, $j \in \{1,\dots,d\}$, be the corresponding multivariate ranks. Then, according to \cite[Section 3.2]{KirSegTsu21}, a $d$-dimensional random vector $\bm V$ with d.f.\ $D$ can be obtained by computing $I = \ip{n U_{d+1}} + 1$ and $V_j = \Beta_{n,r_{Ij}}^{-1}(U_j)$, $j \in \{1,\dots,d\}$, where, as in~\eqref{eq:C:1n:beta}, $\Beta_{n,q}$ denotes the d.f.\ of the distribution Beta$(q, n+1-q)$. Finally, the random vector $\bm W_{\bm u}^{\bm x}$ can again be defined by $(H_1^{-1}(V_1), \dots, H_d^{-1}(V_d))$. More generally, as soon as there exists a method to generate random variates from the copula $D$, it is likely that we will be able to define the function $\Wc_{\bm u}^{\bm x}$ transforming $\bm U$ into the random vector $\bm W_{\bm u}^{\bm x}$ with law $\nu_{\bm u}^{\bm x}$. 
\end{remark}

\begin{remark} 
  \label{rem:Bernstein}
  Let $n, m_1, \dots, m_d \in \N$. For any $\bm x \in (\R^d)^n$ and $\bm u \in [0,1]^d$, let $\nu_{\bm u}^{\bm x}$ be the law of the random vector $(S_{m_1,1,u_1}/m_1,\dots, S_{m_d,d,u_d}/m_d)$, where $S_{m_1,1,u_1},\dots S_{m_d,d,u_d}$ are independent random variables and, for each $j \in \{1,\dots d\}$, $S_{m_j,j,u_j}$ is Binomial$(m_j,u_j)$. Then, from Section~3 in \cite{SegSibTsu17}, $C_{1:n}^\nu$ corresponds to the empirical Bernstein copula with polynomial degrees $m_1, \dots, m_d$. Since, from Lemma 2.6 in \cite{SegSibTsu17}, the empirical beta copula is the empirical Bernstein copula with $m_1 = n, \dots, m_d = n$, it is obviously a particular case of $C_{1:n}^\nu$. It is obtained when, for any $\bm x \in (\R^d)^n$ and $\bm u \in [0,1]^d$, $\nu_{\bm u}^{\bm x}$ is defined as the measure $\mu_{n,\bm u}$ appearing in~\eqref{eq:C:1n:beta:mix} which is the law of the random vector $(S_{n,1,u_1}/n,\dots, S_{n,d,u_d}/n)$. 
\end{remark}

As we shall see in Section~\ref{sec:specific}, when defining some specific members of the above general class of smooth copulas, we will consider the possibly random smoothing distributions $\nu_{\bm u}^{\sss \Xc_{1:n}}$, $\bm u \in [0,1]^d$, to be either all discrete or all continuous with the understanding that, because of~\eqref{eq:var:W}, if $u_j \in \{0,1\}$, the $j$th component of $\bm W_{\bm u}^{\sss \Xc_{1:n}}$ will be degenerate (non-random). 

For any $n \in \N$, $\bm x \in (\R^d)^n$, $\bm r \in [0,n]^d$ and $\bm u \in [0,1]^d$, let
\begin{equation}
  \label{eq:K}
  \Kc_{\bm r}^{\bm x}(\bm u) =  \int_{[0,1]^d} \1(\bm r / n \leq \bm w) \dd\nu_{\bm u}^{\bm x}(\bm w) = \Ex \left \{ \1(\bm r / n \leq \bm W_{\bm u}^{\bm x}) \right\} = \Gc_{\bm u}^{\bm x}(\bm r / n),
\end{equation}
where $\Gc_{\bm u}^{\bm x}(\bm w) = \Pr(\bm W_{\bm u}^{\bm x} \geq \bm w)$, $\bm w \in [0,1]^d$. Note that this implies that, almost surely, $\Gc_{\bm u}^{\sss \Xc_{1:n}}(\bm w) = \Pr(\bm W_{\bm u}^{\sss \Xc_{1:n}} \geq \bm w \mid \Xc_{1:n})$, $\bm w \in [0,1]^d$. By linearity of the integral, with  probability 1, we can then express $C_{1:n}^\nu$ as
\begin{equation}
  \label{eq:C:1n:K}
C_{1:n}^\nu(\bm u) = \frac{1}{n} \sum_{i=1}^n \Kc_{\sss{\bm R^{1:n}_i}}^{\sss \Xc_{1:n}}(\bm u), \qquad \bm u \in [0,1]^d.
\end{equation}


\section{Properties and general form of the smooth estimators}
\label{sec:properties}

In this section, we provide conditions under which smooth empirical copulas of the form~\eqref{eq:C:1n:nu} have standard uniform univariate margins and then conditions under which they are multivariate d.f.s.

\subsection{Univariate margins of the smooth estimators}

We start by investigating the univariate margins of the studied nonparametric copula estimators. As already hinted at in the introduction, the following simple condition plays an important role.

\begin{cond}[No ties]
\label{cond:no:ties}
With  probability 1, there are no ties in each of the component samples $X_{1j}, \dots, X_{nj}$, $j \in \{1,\dots,d\}$, of $\Xc_{1:n}$.
\end{cond}

As verified in \cite{SegSibTsu17}, under this condition, the empirical beta copula $C_{1:n}^\Beta$ defined in~\eqref{eq:C:1n:beta} has standard uniform margins. In the rest of the paper, for any $j \in \{1,\dots,d\}$ and any $\bm{u} \in [0, 1]^d$, let $\bm{u}^{(j)}$ be the vector of $[0, 1]^d$ defined by $u^{(j)}_i = u_j$ if $i = j$ and 1 otherwise. With this notation, the property of having standard uniform univariate margins can be simply written as $C_{1:n}^\Beta(\bm u^{(j)}) = u_j$ for all $j \in \{1,\dots,d\}$ and $\bm u \in [0,1]^d$. For the smooth empirical copula $C_{1:n}^\nu$ defined in~\eqref{eq:C:1n:nu}, we have the following result proven in Appendix~\ref{proofs:short}.

\begin{prop}[Univariate margins of $C_{1:n}^\nu$]
  \label{prop:unif:margin}
  Under Condition~\ref{cond:no:ties}, for any $j \in \{1,\dots,d\}$ and $\bm u \in [0,1]^d$, with  probability 1,
  $$
  C_{1:n}^\nu(\bm u^{(j)}) =  \Ex \left(  \frac{\ip{n W_{j,u_j}^{\sss \Xc_{1:n}}}}{n} \mid \Xc_{1:n} \right).
  $$
\end{prop}

Since $\sup_{w \in [0,1]} | \ip{n w} / n  - w| \to 0$ as $n \to \infty$ and $\Ex \left(  W_{j,u_j}^{\sss \Xc_{1:n}} \mid \Xc_{1:n} \right) = u_j$ almost surely by construction, under Condition~\ref{cond:no:ties}, a smooth empirical copula $C_{1:n}^\nu$ will at least have standard uniform margins asymptotically. Actually, it is easy to verify that, under Condition~\ref{cond:no:ties}, $C_{1:n}^\nu$ will have standard uniform margins if and only if the following condition is satisfied.

\begin{cond}[Condition for uniform margins]
  \label{cond:unif:marg}
For any $\bm x \in (\R^d)^n$, $\bm u \in [0,1]^d$ and $j \in \{1,\dots,d\}$, $W_{j,u_j}^{\bm x}$ takes its values in the set $\{0,1/n,\dots,(n-1)/n,1\}$.
\end{cond}

The previous condition is for instance satisfied when, for any $\bm x \in (\R^d)^n$, $\bm u \in [0,1]^d$ and $j \in \{1,\dots,d\}$, $W_{j,u_j}^{\bm x} = S_{n,j,u_j} / n$, where $S_{n,j,u_j}$ is Binomial$(n,u_j)$. Notice that, when $u_j \in (0,1)$ and $W_{j,u_j}^{\sss \Xc_{1:n}}$ does not take all its values in the set $\{0,1/n,\dots,(n-1)/n,1\}$, one has  with  probability 1 that
\begin{equation*}
  C_{1:n}^\nu(\bm u^{(j)}) = \Ex \left(  \frac{\ip{n W_{j,u_j}^{\sss \Xc_{1:n}}}}{n} \mid \Xc_{1:n} \right)  < \Ex \left(  W_{j,u_j}^{\sss \Xc_{1:n}} \mid \Xc_{1:n} \right) = u_j
\end{equation*}
since $\ip{nw} / n < w$ for all $w \in [0,1] \setminus \{0,1/n,\dots,(n-1)/n,1\}$. Hence, in that case, from a marginal perspective, $C_{1:n}^\nu$ will systematically underestimate $C$.

One possible remedy to this situation is to carry out an asymptotically negligible correction consisting of using $\Kc_{\sss{\bm R_i^{1:n} - \mathbf{a}}}^{\sss \Xc_{1:n}}$ instead of $\Kc_{\sss{\bm R_i^{1:n}}}^{\sss \Xc_{1:n}}$ in~\eqref{eq:C:1n:K}, where $\mathbf{a} = (a,\dots,a) \in (0,1)^d$. Indeed, mimicking the proof of Proposition~\ref{prop:unif:margin}, one obtains that, under Condition~\ref{cond:no:ties}, for any $j \in \{1,\dots,d\}$ and $\bm u \in [0,1]^d$, $C_{1:n}^\nu(\bm u^{(j)}) = \Ex (\ip{n W_{j,u_j}^{\sss \Xc_{1:n}} + a} / n \mid \Xc_{1:n})$ almost surely. It then seems reasonable to choose $a = 1/2$ given that $\sup_{w \in [0,1]} |\ip{n w + a} / n - w| = \max(a, 1 - a)/n$ is minimized for $a = 1/2$ and $\Ex (W_{j,u_j}^{\sss \Xc_{1:n}} \mid \Xc_{1:n}) = u_j$ almost surely by construction. For this reason, from now on, when using smoothing distributions for which Condition~\ref{cond:unif:marg} is not satisfied, instead of $C_{1:n}^\nu$ in~\eqref{eq:C:1n:K}, we will consider the asymptotically equivalent estimator defined by
\begin{equation}
  \label{eq:C:1n:K:cont}
C_{1:n}^{\nu,\text{cor}}(\bm u) = \frac{1}{n} \sum_{i=1}^n \Kc_{\bm R^{1:n}_i - \mathbf{1/2}}^{\sss \Xc_{1:n}}(\bm u), \qquad \bm u \in [0,1]^d.
\end{equation}

\begin{remark}
Some additional thinking reveals that~\eqref{eq:C:1n:K:cont} would have been equivalently obtained if, in~\eqref{eq:C:1n:nu}, the classical empirical copula $C_{1:n}$ in~\eqref{eq:C:1n} were replaced by the empirical d.f.\ of the modified pseudo-observations $\tilde{\bm U}^{1:n}_i = (\bm R^{1:n}_i - \mathbf{1/2})/n$, $i \in \{1,\dots,n\}$. The latter asymptotically equivalent definition of the empirical copula was for instance considered in \cite[Section~5.10.1]{Joe14} and has univariate margins that are uniformly closer to the d.f.\ of the standard uniform distribution than $C_{1:n}$ in~\eqref{eq:C:1n}.  
\end{remark}

\subsection{General form of the smooth estimators}

As already mentioned, we will choose the possibly random smoothing distributions $\nu_{\bm u}^{\sss \Xc_{1:n}}$, $\bm u \in [0,1]^d$, to be either all discrete or all continuous with the understanding that, because of~\eqref{eq:var:W}, if $u_j \in \{0,1\}$, the $j$th component of $\bm W_{\bm u}^{\sss \Xc_{1:n}}$ will be degenerate. In the discrete case, we will further impose Condition~\ref{cond:unif:marg} so that, from Proposition~\ref{prop:unif:margin}, the corresponding smooth empirical copulas of the form~\eqref{eq:C:1n:K} have standard uniform margins. This property will not hold if the smoothing distributions are chosen continuous. Considering in that case the asymptotically equivalent definition in~\eqref{eq:C:1n:K:cont} will however make the uniform distance between the univariate margins of the smooth estimator and the d.f.\ of the standard uniform distribution smaller than~$1/(2n)$. 

Notice that the expressions of $C_{1:n}^\nu$ in~\eqref{eq:C:1n:K} and $C_{1:n}^{\nu,\text{cor}}$ in~\eqref{eq:C:1n:K:cont} both depend on the quantity $\Kc_{\bm r}^{\bm x}$, $\bm x \in (\R^d)^n$, $\bm r \in [0,n]^d$, defined in~\eqref{eq:K}. Specifically, recall that, for any $\bm x \in (\R^d)^n$, $\bm r \in [0,n]^d$ and $\bm u \in [0,1]^d$, $\Kc_{\bm r}^{\bm x}(\bm u) = \Gc_{\bm u}^{\bm x}(\bm r / n)$, where $\Gc_{\bm u}^{\bm x}(\bm w) = \Pr(\bm W_{\bm u}^{\bm x} \geq \bm w)$, $\bm w \in [0,1]^d$. For any $\bm x \in (\R^d)^n$ and $\bm u \in [0,1]^d$, let $\bar \Fc_{\bm u}^{\bm x}(\bm w) = \Pr(\bm W_{\bm u}^{\bm x} > \bm w)$, $\bm w \in [0,1]^d$, be the survival function of $\bm W_{\bm u}^{\bm x}$ and note that $\Gc_{\bm u}^{\bm x}(\bm w)= \bar \Fc_{\bm u}^{\bm x}(\bm w)$, $\bm w \in [0,1]^d$, if $\bm W_{\bm u}^{\bm x}$ is continuous and that $\Gc_{\bm u}^{\bm x}(\bm r / n)= \bar \Fc_{\bm u}^{\bm x}\{ (\bm r - \bm 1) / n\}$, $\bm r \in \{1,\dots,n\}^d$, where $\1 = (1,\dots,1) \in \R^d$ if $\bm W_{\bm u}^{\bm x}$ takes its values in the set $\{0,1/n,\dots,(n-1)/n,1\}^d$. Interestingly enough, for any $\bm x \in (\R^d)^n$ and $\bm u \in [0,1]^d$, we can additionally use the fact that, from Sklar's theorem for survival functions \cite[see, e.g.,][Section~2.5]{HofKojMaeYan18}, there exists a copula $\bar \Cc_{\bm u}^{\bm x}$ (called a survival copula of~$\bm W_{\bm u}^{\bm x}$) such that
\begin{equation*}
\bar \Fc_{\bm u}^{\bm x}(\bm w) = \bar \Cc_{\bm u}^{\bm x} \{\bar \Fc_{1,u_1}^{\bm x}(w_1), \dots, \bar \Fc_{d,u_d}^{\bm x}(w_d) \}, \qquad \bm w \in [0,1]^d,
\end{equation*}
where, for any $j \in \{1,\dots,d\}$, $\bar \Fc_{j,u_j}^{\bm x}(w) = \Pr(W_{j,u_j}^{\bm x} > w)$, $w \in [0,1]$, is the $j$th univariate margin of $\bar \Fc_{\bm u}^{\bm x}$. Note that $\bar \Cc_{\bm u}^{\bm x}$ is not uniquely defined unless $\bm W_{\bm u}^{\bm x}$ is a continuous random vector and that $\bar \Cc_{\bm u}^{\bm x}$ is not a copula of the random vector $\bm W_{\bm u}^{\bm x}$ but a copula of the random vector $-\bm W_{\bm u}^{\bm x}$.

Combining the previous elements, one has that, if the smoothing distributions satisfy Condition~\ref{cond:unif:marg}, for any $\bm x \in (\R^d)^n$, $\bm r \in [1,n]^d$ and $\bm u \in [0,1]^d$, $\Kc_{\bm r}^{\bm x}(\bm u)$ in~\eqref{eq:K} can be expressed as
\begin{equation}
  \label{eq:K:dis}
  \Kc_{\bm r}^{\bm x}(\bm u) = \bar \Cc_{\bm u}^{\bm x} \left[ \bar \Fc_{1,u_1}^{\bm x}\{ (r_1 - 1) / n\}, \dots, \bar \Fc_{d,u_d}^{\bm x} \{ (r_d - 1) / n \} \right],
\end{equation}
which implies that, with  probability 1, the smooth empirical copula $C_{1:n}^\nu$ in~\eqref{eq:C:1n:K} can be expressed, for any $\bm u \in [0,1]^d$, as
\begin{equation}
  \label{eq:C:1n:dis}
  C_{1:n}^\nu(\bm u) = \frac{1}{n} \sum_{i=1}^n \bar \Cc_{\bm u}^{\sss \Xc_{1:n}} \left[ \bar \Fc_{1,u_1}^{\sss \Xc_{1:n}}\{ (R_{i1}^{1:n} - 1) / n\}, \dots, \bar \Fc_{d,u_d}^{\sss \Xc_{1:n}} \{ (R_{id}^{1:n} - 1) / n \} \right],
\end{equation}
whereas, if the smoothing distributions are continuous, as already mentioned, it is better from a marginal perspective to consider the asymptotically equivalent estimator $C_{1:n}^{\nu,\text{cor}}$ in~\eqref{eq:C:1n:K:cont} which, with  probability 1, can be expressed, for any $\bm u \in [0,1]^d$, as
\begin{equation}
  \label{eq:C:1n:cont}
C_{1:n}^{\nu,\text{cor}}(\bm u) = \frac{1}{n} \sum_{i=1}^n \bar \Cc_{\bm u}^{\sss \Xc_{1:n}} \left[ \bar \Fc_{1,u_1}^{\sss \Xc_{1:n}}\{ (R_{i1}^{1:n} - 1/2) / n\}, \dots, \bar \Fc_{d,u_d}^{\sss \Xc_{1:n}} \{ (R_{id}^{1:n} - 1/2) / n \} \right]
\end{equation}
since, for any $\bm x \in (\R^d)^n$, $\bm r \in [1,n]^d$ and $\bm u \in [0,1]^d$, $\Kc_{\bm r - \mathbf{1/2}}^{\bm x}(\bm u)$ can be expressed as
\begin{equation}
  \label{eq:K:cont}
  \Kc_{\bm r - \mathbf{1/2}}^{\bm x}(\bm u) = \bar \Cc_{\bm u}^{\bm x} \left[ \bar \Fc_{1,u_1}^{\bm x}\{ (r_1 - 1/2) / n\}, \dots, \bar \Fc_{d,u_d}^{\bm x} \{ (r_d - 1/2) / n \} \right].
\end{equation}

\subsection{Conditions for being multivariate d.f.s}

From a finite-sample perspective, it seems desirable to focus on estimators that are multivariate d.f.s and thus, possibly, genuine copulas. The following conditions are sufficient for that matter.

\begin{cond}[Condition on the smoothing survival copulas]
  \label{cond:smooth:cop}
  For any $\bm x \in (\R^d)^n$ and $\bm u \in [0,1]^d$, the copulas $\bar \Cc_{\bm u}^{\bm x}$ in~\eqref{eq:K:dis} and~\eqref{eq:K:cont} do not depend on $\bm u$.
\end{cond}

\begin{cond}[Condition on the discrete smoothing survival margins]
  \label{cond:smooth:surv:marg}
  For any $\bm x \in (\R^d)^n$, $j \in \{1,\dots,d\}$ and $w \in [0,1)$, the function $t \mapsto \bar \Fc_{j,t}^{\bm x}(w)$ is right-continuous and increasing on $[0,1]$.
\end{cond}


The following result is proven in Appendix~\ref{proofs:short}.

\begin{prop}[$C_{1:n}^\nu$ is a multivariate d.f.]
  \label{prop:mult:df}
Assume that Conditions~\ref{cond:smooth:cop} and~\ref{cond:smooth:surv:marg} hold. Then, the smooth empirical copula $C_{1:n}^\nu$ in~\eqref{eq:C:1n:dis} is a multivariate d.f. 
\end{prop}

\begin{cor}[$C_{1:n}^\nu$ is a genuine copula]
  \label{cor:copula}
Assume that Conditions~\ref{cond:no:ties},~\ref{cond:unif:marg},~\ref{cond:smooth:cop} and~\ref{cond:smooth:surv:marg} hold. Then, the smooth empirical copula $C_{1:n}^\nu$ in~\eqref{eq:C:1n:nu} is a genuine copula.
\end{cor}

In the case of continuous smoothing distributions, we consider the following analog of Condition~\ref{cond:smooth:surv:marg}.

\begin{cond}[Condition on the continuous smoothing survival margins]
  \label{cond:smooth:surv:marg:cont}
  For any $\bm x \in (\R^d)^n$, $j \in \{1,\dots,d\}$ and $w \in (0,1)$, the function $t \mapsto \bar \Fc_{j,t}^{\bm x}(w)$ is right-continuous and increasing on $[0,1]$.
\end{cond}

The following result is proven in Appendix~\ref{proofs:short}.

\begin{prop}[$C_{1:n}^{\nu,\text{cor}}$ is a multivariate d.f.]
  \label{prop:mult:df:cont}
Assume that Conditions~\ref{cond:smooth:cop} and~\ref{cond:smooth:surv:marg:cont} hold. Then, the smooth empirical copula $C_{1:n}^{\nu,\text{cor}}$ in~\eqref{eq:C:1n:cont} is a multivariate d.f.
\end{prop}

\section{Estimators based on smoothing distributions with scaled binomial, scaled beta-binomial or beta margins}
\label{sec:specific}

In order to define specific smooth empirical copulas of the form~\eqref{eq:C:1n:dis} or~\eqref{eq:C:1n:cont} that are multivariate d.f.s, we start by making three proposals for the smoothing survival margins. We first make two proposals for the (discrete) smoothing survival margins of the estimator~\eqref{eq:C:1n:dis} for which Condition~\ref{cond:smooth:surv:marg} is satisfied and then one proposal for the (continuous) smoothing survival margins of the estimator~\eqref{eq:C:1n:cont} for which Condition~\ref{cond:smooth:surv:marg:cont} is satisfied. We end this section by discussing the choice of the survival copula of the smoothing distributions.

\subsection{Scaled binomial and beta-binomial smoothing survival margins}
\label{sec:surv:marg:dis}

Having Corollary~\ref{cor:copula} in mind for the estimator~\eqref{eq:C:1n:dis} and given $u \in [0,1]$, we wish to define a random variable $W$ that takes its values in the set $\{0,1/n,\dots,(n-1)/n,1\}$ such that $\Ex(W) = u$. From~\eqref{eq:var:W}, only the case $u \in (0,1)$ actually needs to be dealt with. Following~\cite{SegSibTsu17}, it is natural to attempt to start from a random variable $S$ that takes its values in $\{0,1,\dots,n\}$ and to set $W = S/n$. A first straightforward choice due to~\cite{SegSibTsu17} is to take $S$ to be Binomial$(n,u)$. As already mentioned, this immediately leads to $W$ taking its values in the set $\{0,1/n,\dots,(n-1)/n,1\}$ and satisfying $\Ex(W) = u$ and $\Var(W) = u(1-u)/n$.

As an alternative distribution for the random variable $S$ taking its values in $\{0,1,\dots,n\}$, we investigate next the possibility of considering the (more dispersed) Beta-Binomial$(n, \alpha, \beta)$ whose shape parameters $\alpha > 0$ and $\beta > 0$ remain to be specified. Recall that the probability mass function of the latter distribution is given by
\begin{equation*}
  \Pr(S = s) = \binom{n}{s} \frac{B(s + \alpha,n-s+\beta)}{B(\alpha,\beta)}, \qquad s \in \{0,1,\dots,n\}, 
\end{equation*}
where $B$ is the Beta function.

We start from the fact that, since $S$ is Beta-Binomial$(n, \alpha, \beta)$,
\begin{align}
  \label{eq:beta-binomial:ex}
  \Ex(W) &= \frac{\alpha}{\alpha + \beta}, \\
  \label{eq:beta-binomial:var}
  \Var(W) &= \frac{\alpha \beta (\alpha + \beta + n)}{n(\alpha + \beta)^2 (\alpha + \beta + 1)}. 
\end{align}
Since the expectation of $W$ is required to be $u$ by construction, we immediately obtain from~\eqref{eq:beta-binomial:ex} that
\begin{equation}
  \label{eq:beta}
  \beta = \frac{\alpha(1-u)}{u},
\end{equation}
which, combined with~\eqref{eq:beta-binomial:var}, allows us to rewrite the variance of $W$ as
\begin{equation*}
  \Var(W) = \frac{u(1 - u)(\alpha + un)}{n(\alpha + u)} = \frac{u(1 - u)}{n} \times \frac{\alpha + un}{\alpha + u}.
\end{equation*}
The variance of $W$ is thus the variance of $W$ if $S$ were Binomial$(n,u)$ multiplied by the factor $(\alpha + u n)/(\alpha + u)$. Ideally, we would want to control how much more dispersed scaled beta-binomial margins are compared to the corresponding scaled binomial margins. To do so, we set the latter factor to be a constant $\rho > 1$, that is,
$$
\frac{\alpha + un}{\alpha + u} = \rho
$$
and attempt to solve for $\alpha > 0$. Provided that $\rho < n$, we obtain that $\alpha = u(n - \rho) / (\rho - 1)$ and then, using~\eqref{eq:beta}, that $\beta = (1-u)(n - \rho) / (\rho - 1)$.

In summary, as an alternative distribution for the discrete random variable $S$, we consider the Beta-Binomial$\big(n,u(n - \rho)/(\rho - 1), (1-u)(n - \rho)/(\rho - 1) \big)$, where $\rho \in (1,n)$ is an additional parameter. As required, we have that  $\Ex(W) = u$. Furthermore,
\begin{equation}
  \label{eq:beta-binomial:var:2}
  \Var(W) = \rho \times \frac{u(1-u)}{n},
\end{equation}
so that the factor $\rho$ describes how much more dispersed a scaled beta-binomial smoothing survival margin will be compared to the corresponding scaled binomial.

Notice that the probability mass functions of the Binomial$(n,u)$ and the Binomial$(n,1 - u)$ are symmetrical with respect to $n/2$, and that, for any  $\rho \in (1,n)$, the probability mass functions of the Beta-Binomial$\big(n,u(n - \rho)/(\rho - 1), (1-u)(n - \rho)/(\rho - 1) \big)$ and the Beta-Binomial$\big(n,(1-u)(n - \rho)/(\rho - 1), u(n - \rho)/(\rho - 1) \big)$ are also symmetrical with respect to $n/2$. This implies that, with respect to one coordinate, the smoothing around $u$ will be the ``reflection'' of the smoothing around $1-u$.

For any $u \in [0,1]$, let $\bar B_{n,u}$ be the survival function of the Binomial$(n,u)$ and, for any $u \in [0,1]$ and $\rho \in (1,n)$, let $\bar \Bc_{n,u,\rho}$ be the survival function of the Beta-Binomial$\big( n,u(n - \rho)/(\rho - 1), (1-u)(n - \rho)/(\rho - 1) \big)$. Two subclasses of the class of smooth empirical copulas given by~\eqref{eq:C:1n:dis} that, under Condition~\ref{cond:smooth:cop}, still depend on the choice of the possibly data-dependent survival copula $\bar \Cc = \bar \Cc^{\sss \Xc_{1:n}}$, can thus be defined, for any $\bm u \in [0,1]^d$, by
\begin{align}
  \label{eq:ec:Bin:marg}
  C_{1:n}^{\sss{\bar B, \bar \Cc}}(\bm u) = \frac{1}{n} \sum_{i=1}^n \bar \Cc^{\sss \Xc_{1:n}} \left\{ \bar B_{n, u_1}(R_{i1}^{1:n} - 1), \dots, \bar B_{n, u_d}(R_{id}^{1:n} - 1) \right\}, \\
  \label{eq:ec:BetaBin:marg}
  C_{1:n}^{\sss{\bar \Bc, \bar \Cc}}(\bm u) = \frac{1}{n} \sum_{i=1}^n \bar \Cc^{\sss \Xc_{1:n}} \left\{ \bar \Bc_{n, u_1, \rho}(R_{i1}^{1:n} - 1), \dots, \bar \Bc_{n, u_d, \rho}(R_{id}^{1:n} - 1) \right\},
\end{align} 
respectively. It is of course possible to imagine a version of the second estimator for which the common additional parameter $\rho$ of the $d$ scaled beta-binomial smoothing survival margins depends on the data, that is, $\rho = \rho^{\sss \Xc_{1:n}}$ (subject to the constraint that $\rho^{\sss \Xc_{1:n}} \in (1,n)$ almost surely). By construction, the smoothing survival margins satisfy Condition~\ref{cond:unif:marg}, which, according to Proposition~\ref{prop:unif:margin}, implies that $C_{1:n}^{\sss{\bar B,\bar \Cc}}$ and $C_{1:n}^{\sss{\bar \Bc,\bar \Cc}}$ have standard uniform univariate margins. As far as Condition~\ref{cond:smooth:surv:marg} is concerned, we have the following results proven in Appendix~\ref{proofs:cond:marg}.

\begin{prop}
  \label{prop:Bin:cond}
Condition~\ref{cond:smooth:surv:marg} is satisfied with $\bar \Fc_{j,t}^{\bm x}$ defined by $\bar \Fc_{j,t}^{\bm x}(w) = \bar B_{n,t}(n w)$, $w \in [0,1]$.  Specifically, for any $w \in [0,1)$, the function $t \mapsto \bar B_{n,t}(n w)$ is continuous and increasing on $[0,1]$.
\end{prop}

\begin{prop}
  \label{prop:BetaBin:cond}
Condition~\ref{cond:smooth:surv:marg} is satisfied with $\bar \Fc_{j,t}^{\bm x}$ defined by $\bar \Fc_{j,t}^{\bm x}(w) = \bar \Bc_{n,t,\rho}(nw)$, $w \in [0,1]$, for any $\rho \in (1,n)$. Specifically, for any $\rho \in (1,n)$ and $w \in [0,1)$, the function $t \mapsto \bar \Bc_{n,t,\rho}(n w)$ is continuous and increasing on $[0,1]$.
\end{prop}

Hence, according to Proposition~\ref{prop:mult:df}, the estimators $C_{1:n}^{\sss{\bar B,\bar \Cc}}$ and $C_{1:n}^{\sss{\bar \Bc,\bar \Cc}}$ are multivariate d.f.s and therefore, according to Corollary~\ref{cor:copula}, genuine copulas \ik{if Condition~\ref{cond:no:ties} holds}. Notice also that the smooth empirical copula $C_{1:n}^{\sss{\bar B, \bar \Cc}}$ coincides with the empirical beta copula $C_{1:n}^\Beta$ in~\eqref{eq:C:1n:beta} if $\bar \Cc^{\sss \Xc_{1:n}}$ is taken to be the independence copula $\Pi$. Based on~\eqref{eq:ec:Bin:marg}, an alternative notation for the empirical beta copula is thus $C_{1:n}^{\sss{\bar B,\Pi}}$. 

\begin{remark}
  Let $m \in \N$ and recall from Remark~\ref{rem:Bernstein} that if, for any $\bm x \in (\R^d)^n$ and $\bm u \in [0,1]^d$, $\nu_{\bm u}^{\bm x}$ is the law of the random vector $(S_{m,1,u_1}/m,\dots, S_{m,d,u_d}/m)$ where $S_{m,1,u_1},\dots S_{m,d,u_d}$ are independent random variables such that, for each $j \in \{1,\dots d\}$, $S_{m,j,u_j}$ is Binomial$(m,u_j)$, then, from Section~3 in \cite{SegSibTsu17}, $C_{1:n}^\nu$ in~\eqref{eq:C:1n:nu} is the empirical Bernstein copula with polynomial degrees all equal to $m$. It immediately follows that the expectation and variance of the $j$th margin of the underlying smoothing distribution corresponding to $\bm u \in [0,1]^d$ are $u_j$ and $u_j (1-u_j) / m$, respectively. Let $\rho \in (1,n)$ and let us focus on the special case $m = \up{n/\rho}$. The underlying smoothing distribution corresponding to $\bm u \in [0,1]^d$ can then be regarded as close to the smoothing distribution corresponding to the same $\bm u$ involved in the definition of the estimator $C_{1:n}^{\sss{\bar \Bc, \Pi}}$ generically defined in \eqref{eq:ec:BetaBin:marg} with dispersion parameter $\rho$. Indeed, both smoothing distributions have independent copula, expectation $\bm u$ and the variances of their $j$th margin are $u_j (1-u_j) / \up{n/\rho}$ and $\rho u_j (1-u_j) / n$, respectively (see~\eqref{eq:beta-binomial:var:2}). The estimator $C_{1:n}^{\sss{\bar \Bc, \Pi}}$ with dispersion parameter $\rho$ can however be regarded as an advantageous replacement of the empirical Bernstein copula with polynomial degrees all equal to $\up{n/\rho}$. Indeed, \ik{if Condition~\ref{cond:no:ties} holds,} the former is a genuine copula by Corollary~\ref{cor:copula} (no matter how $\rho \in (1,n)$ is chosen) whereas the latter is a genuine copula only when $n$ is a multiple of $\up{n/\rho}$ by Proposition~2.5 in \cite{SegSibTsu17}. When $\rho = 3$, the disadvantages in terms of finite-sample performance of the empirical Bernstein copula with polynomial degrees all equal to $\up{n/\rho}$ are highlighted in Fig.~3 of \cite{SegSibTsu17}.
\end{remark}

\subsection{Beta smoothing survival margins}
\label{sec:beta:margin}

Since the support of the beta distribution is included in $[0,1]$, it is natural to attempt to use it to define the continuous smoothing survival margins involved in the definition of the estimator~\eqref{eq:C:1n:cont}. Thus, given $u \in [0,1]$, let $W$ be a random variable with the beta distribution Beta$(\alpha, \beta)$ whose shape parameters $\alpha > 0$ and $\beta > 0$ need to be determined so that $\Ex(W) = u$. Again, from~\eqref{eq:var:W}, only the case $u \in (0,1)$ needs to be addressed. We know that
\begin{align}
  \label{eq:beta:ex}
  \Ex(W) &= \frac{\alpha}{\alpha + \beta}, \\
  \label{eq:beta:var}
  \Var(W) &= \frac{\alpha \beta}{(\alpha + \beta)^2 (\alpha + \beta + 1)}.
\end{align}
From the fact that the expectation of $W$ is required to be $u$ and~\eqref{eq:beta:ex}, we obtain that $\beta = \alpha(1-u)/u$, which, combined with~\eqref{eq:beta:var}, implies in turn that
$$
\Var(W) = \frac{u^2(1-u)}{\alpha + u} = \frac{u(1 - u)}{n} \times \frac{nu}{\alpha + u}.
$$
Proceeding as for the scaled beta-binomial smoothing survival margins but keeping in mind that the variance of the beta distribution with expectation $u$ can be arbitrarily small, we set
$$
\frac{nu}{\alpha + u} = \rho > 0
$$
and attempt to solve for $\alpha > 0$. Provided that $\rho < n$, we obtain that $\alpha = u(n-\rho) / \rho$ and, consequently, that $\beta  = (1-u)(n-\rho) / \rho$. The distribution of $W$ can thus be taken to be the Beta$\big( u(n-\rho)/\rho, (1-u)(n-\rho)/\rho \big)$, where $\rho \in (0,n)$ is an additional parameter. As required, we have that $\Ex(W) = u$, whereas the parameter $\rho$ controls the variance of $W$ which is given by
\begin{equation}
  \label{eq:beta:var:2}
  \Var(W) = \rho \times \frac{u(1-u)}{n}.
\end{equation}
Hence, the variance of $W$ can again be compared to the variance of the corresponding scaled binomial survival margin. Notice however that, unlike for scaled beta-binomial survival margins, beta survival margins can be underdispersed compared to the corresponding scaled binomial survival margins.

From the fact that the densities of the Beta$(\alpha, \beta)$ and the Beta$(\beta,\alpha)$ are symmetrical with respect to 1/2, we again have that, with respect to one coordinate, the smoothing around $u$ will be the ``reflection'' of the smoothing around~$1-u$.

For any $u \in [0,1]$ and $\rho \in (0,n)$, let $\bar \beta_{n, u, \rho}$ be the survival function of the Beta$\big( u(n-\rho)/\rho, (1-u)(n-\rho)/\rho \big)$, $\rho \in (0,n)$. One subclass of the class of smooth empirical copulas~\eqref{eq:C:1n:cont} that, under Condition~\ref{cond:smooth:cop}, still depends on the choice of the possibly data-dependent survival copula $\bar \Cc = \bar \Cc^{\sss \Xc_{1:n}}$, can then be defined by
\begin{equation}
  \label{eq:ec:Beta:marg}
  C_{1:n}^{\sss{\bar \beta,\bar \Cc}}(\bm u) = \frac{1}{n} \sum_{i=1}^n \bar \Cc^{\sss \Xc_{1:n}} \left[ \bar \beta_{n, u_1, \rho}\{(R_{i1}^{1:n}-1/2)/n\}, \dots, \bar \beta_{n, u_d, \rho}\{(R_{id}^{1:n}-1/2)/n\} \right], \qquad \bm u \in [0,1]^d.
\end{equation}
As for the estimator~\eqref{eq:ec:BetaBin:marg}, it is possible to imagine a version of~\eqref{eq:ec:Beta:marg} for which the common additional parameter $\rho$ of the $d$ beta smoothing survival margins depends on the data, that is, $\rho = \rho^{\sss \Xc_{1:n}}$ (subject to the constraint that $\rho^{\sss \Xc_{1:n}} \in (0,n)$ almost surely). As far as Condition~\ref{cond:smooth:surv:marg:cont} is concerned, we have the following result proven in Appendix~\ref{proofs:cond:marg}.

\begin{prop}
  \label{prop:Beta:cond}
Condition~\ref{cond:smooth:surv:marg:cont} is satisfied with $\bar \Fc_{j,t}^{\bm x} = \bar \beta_{n,t,\rho}$, for any $\rho \in (0,n)$. Specifically, for any $\rho \in (0,n)$ and $w \in (0,1)$, the function $t \mapsto \bar \beta_{n,t,\rho}(w)$ is continuous and increasing on $[0,1]$.
\end{prop}

Hence, according to Proposition~\ref{prop:mult:df:cont}, $C_{1:n}^{\sss{\bar \beta,\bar \Cc}}$ is a multivariate d.f. \ik{If Condition~\ref{cond:no:ties} holds,}  it is however not a genuine copula since it does not have standard uniform margins. The latter can for instance be verified numerically using the \textsf{R} implementation of the estimator provided on the web page of the first author.

\subsection{On the survival copula of the smoothing distributions}
\label{sec:surv:cop}

In view of Propositions~\ref{prop:mult:df} and~\ref{prop:mult:df:cont}, it seems desirable to choose the survival copulas of the smoothing distributions appearing in~\eqref{eq:C:1n:dis} and~\eqref{eq:C:1n:cont} such that Condition~\ref{cond:smooth:cop} holds. In order to propose a meaningful choice for the resulting possibly data-dependent survival copula~$\bar \Cc^{\sss \Xc_{1:n}}$, we first gather hereafter some theoretical facts.

Assume that Condition~\ref{cond:no:ties} holds and fix $\bm u \in [0,1]^d$. If $\bm W_{\bm u}^{\sss \Xc_{1:n}}$ takes its values in the set $\{0,1/n,\dots,(n-1)/n,1\}^d$, it is easy to verify from~\eqref{eq:C:1n:dis} that
\begin{equation}
  \label{eq:ex:C:1n:nu:dis}
\Ex\{C_{1:n}^\nu(\bm u)\} = \Ex \Big( \bar \Cc^{\sss \Xc_{1:n}} [ \bar \Fc_{1,u_1}^{\sss \Xc_{1:n}}\{ (R_{11}^{1:n} - 1) / n\}, \dots, \bar \Fc_{d,u_d}^{\sss \Xc_{1:n}} \{ (R_{1d}^{1:n} - 1) / n \} ] \Big)
\end{equation}
since $\bm X_1,\dots,\bm X_n$ are identically distributed. Note furthermore that, for any $j \in \{1,\dots,d\}$, 
\begin{equation}
  \label{eq:marg:dis}
  \begin{split}
    \Ex [ \bar \Fc_{j,u_j}^{\sss \Xc_{1:n}}\{ (R_{1j}^{1:n} - 1) / n \}] &= \Ex \left( \Ex \left[ \bar \Fc_{j,u_j}^{\sss \Xc_{1:n}}\{ (R_{1j}^{1:n} - 1) / n \} \mid \Xc_{1:n} \right] \right) = \Ex \left[ \frac{1}{n} \sum_{i=1}^n  \bar \Fc_{j,u_j}^{\sss \Xc_{1:n}}\{ (i - 1) / n \} \right]  \\
    &= \Ex \left[ \frac{1}{n} \sum_{i=1}^n  \Pr \{ W_{j,u_j}^{\sss \Xc_{1:n}} > (i - 1) / n \mid \Xc_{1:n} \}  \right] = \Ex \left[ \frac{1}{n} \sum_{i=1}^n  \Ex \left\{ \1( W_{j,u_j}^{\sss \Xc_{1:n}} \geq i / n) \mid \Xc_{1:n} \right\} \right]  \\
    &= \Ex \left[ \Ex\left\{ \frac{1}{n} \sum_{i=1}^n  \1 ( i \leq n W_{j,u_j}^{\sss \Xc_{1:n}}) \mid \Xc_{1:n} \right\} \right] = \Ex \left\{ \Ex\left( \frac{\ip{n W_{j,u_j}^{\sss \Xc_{1:n}}}}{n} \mid \Xc_{1:n} \right) \right\} \\
    &= \Ex \left\{ \Ex ( W_{j,u_j}^{\sss \Xc_{1:n}} \mid \Xc_{1:n} ) \right\} = \Ex(u_j) = u_j.
  \end{split}
\end{equation}
In other words, the expected value of $C_{1:n}^\nu(\bm u)$ is the expectation of a random variable obtained by applying the copula $\bar \Cc^{\sss \Xc_{1:n}}$ to the components of a random vector with expectation~$\bm u$. Similarly, if $\bm W_{\bm u}^{\sss \Xc_{1:n}}$ is continuous, from~\eqref{eq:C:1n:cont}, we have that
\begin{equation}
  \label{eq:ex:C:1n:nu:cont}
\Ex\{C_{1:n}^{\nu,\text{cor}}(\bm u)\} = \Ex \Big( \bar \Cc^{\sss \Xc_{1:n}} [ \bar \Fc_{1,u_1}^{\sss \Xc_{1:n}}\{ (R_{11}^{1:n} - 1/2) / n\}, \dots, \bar \Fc_{d,u_d}^{\sss \Xc_{1:n}} \{ (R_{1d}^{1:n} - 1/2) / n \} ] \Big)
\end{equation}
and, moreover, proceeding similarly to~\eqref{eq:marg:dis}, that
\begin{equation}
  \label{eq:marg:cont}
  \begin{split}
    \Ex[ \bar \Fc_{j,u_j}^{\sss \Xc_{1:n}}\{ (R_{1j}^{1:n} - 1/2) / n \}] 
    &= \Ex\left( \frac{\ip{n W_{j,u_j}^{\sss \Xc_{1:n}} + 1/2}}{n}  \right),
  \end{split}
\end{equation}
which, combined with the fact that $\sup_{w \in [0,1]} |\ip{n w + 1/2} / n - w| = 1/(2n)$, implies that
\begin{equation}
  \label{eq:marg:bias:cont}
\big| \Ex[ \bar \Fc_{j,u_j}^{\sss \Xc_{1:n}}\{ (R_{1j}^{1:n} - 1/2) / n \} ] - u_j \big| \leq  1/(2n).
\end{equation}
In other words, in the case of a continuous smoothing distribution, the expected value of $C_{1:n}^{\nu,\text{cor}}(\bm u)$ is the expectation of a random variable obtained by applying the copula $\bar \Cc^{\sss \Xc_{1:n}}$ to the components of a random vector whose expectation is within $1/(2n)$ of $\bm u$.

Given that $C_{1:n}^\nu(\bm u)$ and $C_{1:n}^{\nu,\text{cor}}(\bm u)$ are estimators of $C(\bm u)$, the previous derivations suggest that it may be meaningful to expect that, in many situations, their biases will be minimized (or at least will be small) if the survival copula $\bar \Cc^{\sss \Xc_{1:n}}$ of $\bm W_{\bm u}^{\sss \Xc_{1:n}}$ is taken equal to the copula $C$. When $C$ is the independence copula $\Pi$, the above choice is optimal if Condition~\ref{cond:unif:marg} holds. Indeed, setting $\bar \Cc^{\sss \Xc_{1:n}} = \Pi$ in~\eqref{eq:ex:C:1n:nu:dis} and using~\eqref{eq:marg:dis} immediately yields that $C_{1:n}^\nu(\bm u)$ is an unbiased estimator of $C(\bm u) = \Pi(\bm u)$. Similarly, if $\bm W_{\bm u}^{\sss \Xc_{1:n}}$ is continuous, one straightforwardly obtains from~\eqref{eq:ex:C:1n:nu:cont},~\eqref{eq:marg:cont} and~\eqref{eq:marg:bias:cont} that the bias of $C_{1:n}^{\nu,\text{cor}}(\bm u)$ will decrease quickly as $n$ increases.

The previous discussion focusing on the bias thus suggests, if $C$ were known, to choose the survival copula $\bar \Cc^{\sss \Xc_{1:n}}$ in~\eqref{eq:C:1n:dis} and~\eqref{eq:C:1n:cont} equal to $C$. Since $C$ is not known, a natural solution consists of taking $\bar \Cc^{\sss \Xc_{1:n}}$ equal to a pilot estimate of $C$ based on $\Xc_{1:n}$. This aspect will be empirically investigated in Section~\ref{sec:MC:est}.

\section{Suggested data-adaptive alternatives to the empirical beta copula}
\label{sec:MC:est}

The first aim of this section is to study, under Condition~\ref{cond:smooth:cop}, the finite-sample performance of the estimators $C_{1:n}^{\sss{\bar B, \bar \Cc}}$ in~\eqref{eq:ec:Bin:marg}, $C_{1:n}^{\sss{\bar \Bc, \bar \Cc}}$ in~\eqref{eq:ec:BetaBin:marg} and $C_{1:n}^{\sss{\bar \beta, \bar \Cc}}$ in~\eqref{eq:ec:Beta:marg} in the case when the survival copula $\bar \Cc^{\sss \Xc_{1:n}}$ does not depend on the data. These preliminary investigations will then be used to propose a natural way of choosing $\bar \Cc^{\sss \Xc_{1:n}}$ from the data in the expressions of $C_{1:n}^{\sss{\bar B, \bar \Cc}}$, $C_{1:n}^{\sss{\bar \Bc, \bar \Cc}}$ and $C_{1:n}^{\sss{\bar \beta, \bar \Cc}}$. The finite-sample performance of the resulting estimators will be finally studied for many bivariate and trivariate data-generating models.

\subsection{Experiments for understanding the influence of the spread and the shape of the smoothing distributions}
\label{sec:influence}

Following \cite{SegSibTsu17}, for each sample size $n$ under consideration, each data generating copula $C$ and each copula estimator $\hat C_{1:n}$ under investigation, we estimated the following three performance measures:
\begin{align*}
  \text{integrated squared bias: } & \int_{[0,1]^d} \left[ \Ex  \left\{ \hat C_{1:n}(\bm u) - C(\bm u) \right \} \right]^2 \dd \bm u,\\
  \text{integrated variance: } & \int_{[0,1]^d} \Ex \left( \left[ \hat C_{1:n}(\bm u) - \Ex \{ \hat C_{1:n} (\bm u) \} \right]^2 \right) \dd \bm u,\\
  \text{integrated mean squared error: } & \int_{[0,1]^d} \Ex \left[ \left\{ \hat C_{1:n}(\bm u) - C(\bm u) \right\}^2 \right] \dd \bm u.
\end{align*}
To do so,  we proceeded by Monte Carlo simulation using $20{,}000$ independent random samples of size $n$ from $C$ and applied the trick described in detail in Appendix~B of \cite{SegSibTsu17}. All the numerical experiments were carried out using the \textsf{R} statistical environment and its packages \texttt{copula} \cite{copula} and \texttt{extraDistr} \cite{extraDistr}.

\begin{figure}[t!]
\begin{center}
  \includegraphics*[width=1\linewidth]{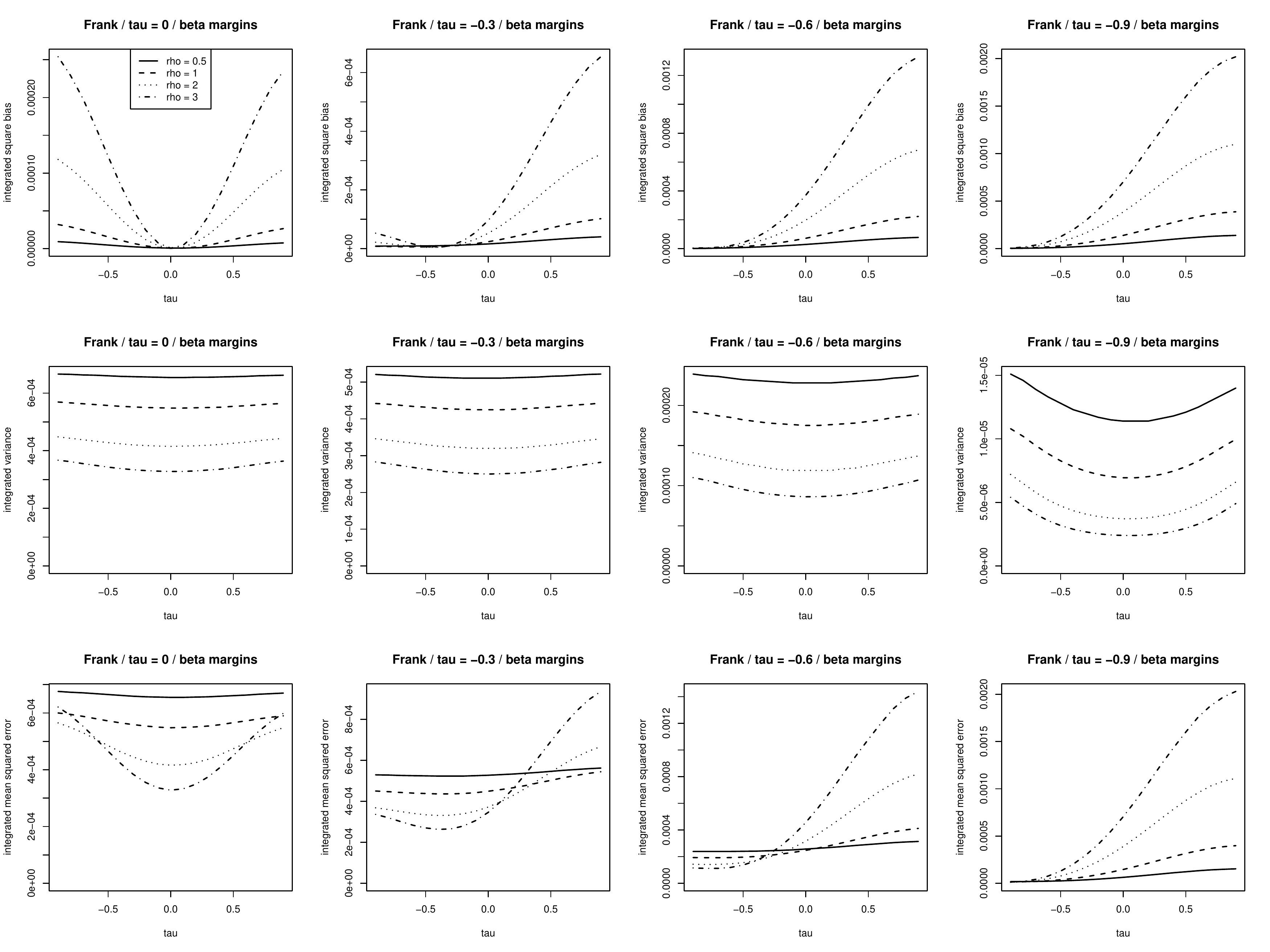}
  \caption{\label{fig:frank} Influence of the dispersion parameter $\rho$ of the beta smoothing survival margins and of the value of Kendall's tau  (given on the $x$-axis of each panel) of the smoothing Frank copula~$\bar \Cc^{\sss \Xc_{1:n}}$ (not taken to depend on the data) on the three performances measures of the smooth estimator $C_{1:n}^{\sss{\bar \beta, \bar \Cc}}$ in~\eqref{eq:ec:Beta:marg} computed from samples of size $n=30$ from a Frank copula with a Kendall's tau in $\{0, -0.3, - 0.6, - 0.9\}$.}
\end{center}
\end{figure}

As a first experiment, we considered $n \in \{10, 30, 50\}$ and $C$ to be the bivariate Frank copula \cite{Gen87} with a Kendall's tau in $\{0, -0.3, - 0.6, - 0.9\}$. We then estimated the above three performance measures for the smooth empirical copula $C_{1:n}^{\sss{\bar \beta, \bar \Cc}}$ in~\eqref{eq:ec:Beta:marg} with $\rho \in \{0.5, 1, 2, 3\}$ and $\bar \Cc^{\sss \Xc_{1:n}}$ the Frank copula with a Kendall's tau in $\{-0.9,-0.8,\dots,0.9\}$ (thus not taken to depend on the data). The results are represented in Fig.~\ref{fig:frank} for $n = 30$ and are qualitatively similar for $n \in \{10,50\}$. Each column of graphs corresponds to a different Kendall's tau for the data generating Frank copula. The first (resp.\ second, third) row of plots reports the integrated squared bias (resp.\ integrated variance, integrated mean squared error). The four curves in each panel represent the four considered values of the dispersion parameter $\rho$ and give the value of the performance measure against the value of Kendall's tau of the (data-independent) Frank copula $\bar \Cc^{\sss \Xc_{1:n}}$ of the smoothing survival distribution.

From the first row of graphs in Fig.~\ref{fig:frank}, we see that, as expected from the discussion in Section~\ref{sec:surv:cop}, the integrated squared bias is minimized when the survival copula $\bar \Cc^{\sss \Xc_{1:n}}$ is close to the data generating copula $C$. The larger the value of the marginal dispersion parameter~$\rho$, the more obvious the previous conclusion is. The second row of plots reveals, on one hand, that the shape of the smoothing distribution controlled by $\bar \Cc^{\sss \Xc_{1:n}}$ has relatively little influence on the integrated variance (although the latter seems always minimized when $\bar \Cc^{\sss \Xc_{1:n}}$ is close to the independence copula $\Pi$) and, on the other hand, that the larger~$\rho$, the lower the integrated variance. The third row of graphs combines the previous conclusions and shows that the integrated mean squared error seems minimized when $\bar \Cc^{\sss \Xc_{1:n}}$ is close to $C$ and $\rho$ is large provided the absolute value of Kendall's tau of $C$ is strictly smaller than 0.9. In the case of strongly (negatively) dependent observations, the lowest integrated mean squared error seems also reached for $\rho = 0.5$, the latter setting having the advantage of being hardly unaffected by the choice of~$\bar \Cc^{\sss \Xc_{1:n}}$.

\begin{figure}[t!]
\begin{center}
  \includegraphics*[width=1\linewidth]{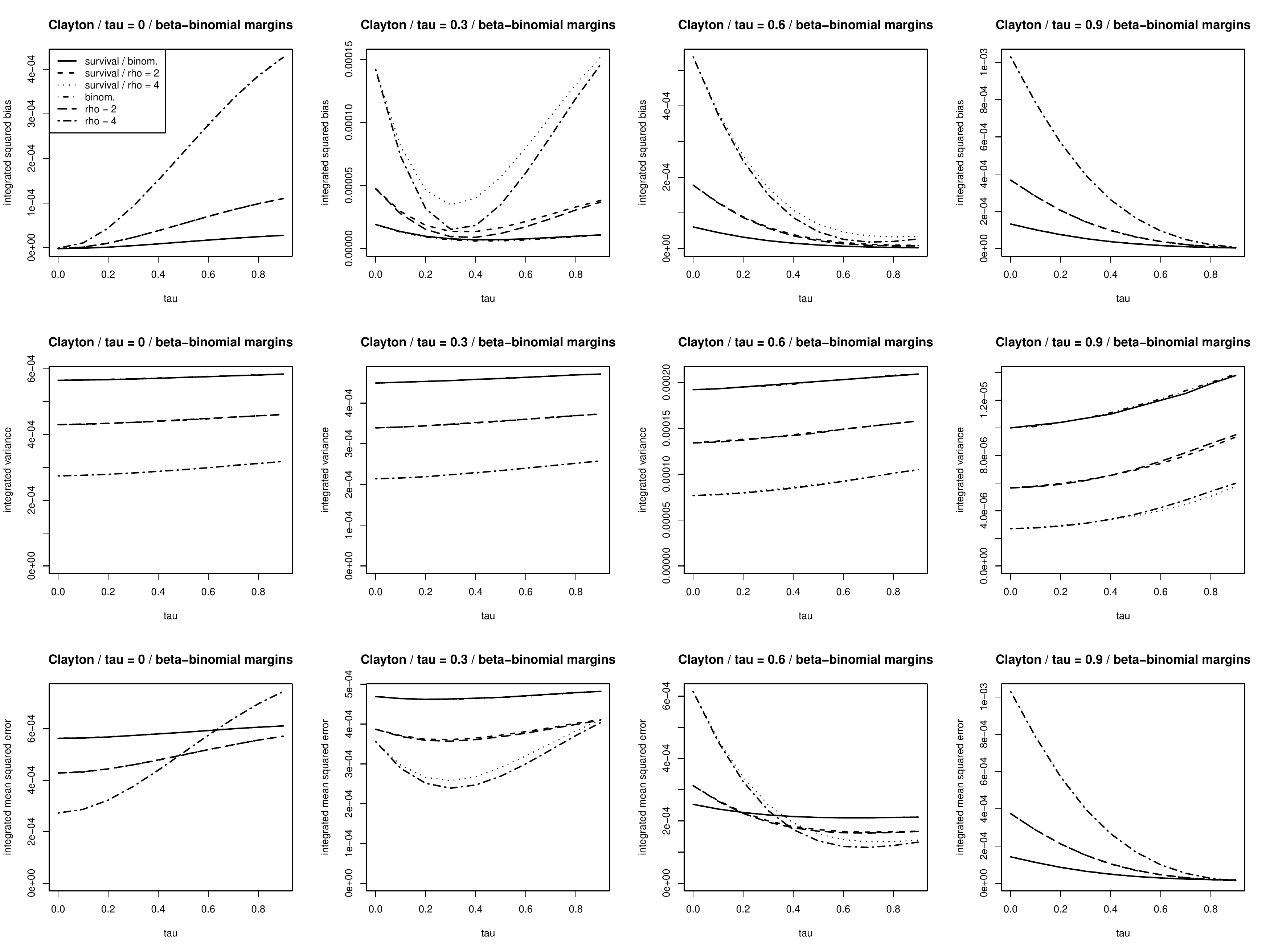}
  \caption{\label{fig:clayton} Influence of the dispersion parameter $\rho$ of the scaled beta-binomial smoothing survival margins and of the value of Kendall's tau  (given on the $x$-axis of each panel) of the smoothing Clayton or survival Clayton copula $\bar \Cc^{\sss \Xc_{1:n}}$ (not taken to depend on the data) on the three performances measures of the smooth estimators $C_{1:n}^{\sss{\bar B, \bar \Cc}}$ in~\eqref{eq:ec:Bin:marg} and $C_{1:n}^{\sss{\bar \Bc, \bar \Cc}}$ in~\eqref{eq:ec:BetaBin:marg} computed from samples of size $n = 30$ drawn from a Clayton copula with a Kendall's tau in $\{0, 0.3, 0.6, 0.9\}$.}
\end{center}
\end{figure}

As a second experiment, we considered a similar setting but with $C$ a bivariate Clayton copula \cite{Cla78} and for the estimators $C_{1:n}^{\sss{\bar B, \bar \Cc}}$ in~\eqref{eq:ec:Bin:marg} and $C_{1:n}^{\sss{\bar \Bc, \bar \Cc}}$ in~\eqref{eq:ec:BetaBin:marg} involving scaled binomial and scaled beta-binomial smoothing survival margins with dispersion parameter $\rho \in \{2, 4\}$, respectively. This time, we also allowed $\bar \Cc^{\sss \Xc_{1:n}}$ to be a survival Clayton copula (still not taken to depend on the data). Note that a Clayton copula and the corresponding survival copula have the same value of Kendall's tau whereas they are very different: the former is lower tail dependent while the latter is upper tail dependent \cite[see, e.g.,][Section~3.4.1]{HofKojMaeYan18}. The results are represented in Fig.~\ref{fig:clayton} for $n = 30$. The conclusions are the same as those drawn after the first experiment. In particular, the curves for $\bar \Cc^{\sss \Xc_{1:n}}$ the Clayton copula and $\bar \Cc^{\sss \Xc_{1:n}}$ the corresponding survival Clayton copula for the same value of $\rho$ in Fig.~\ref{fig:clayton} confirm that the integrated bias and the integrated mean squared error are minimized when $\bar \Cc^{\sss \Xc_{1:n}}$ is approximately equal to $C$. The conclusions remain qualitatively the same for $n \in \{10,50\}$ or when the Clayton family is replaced by the Gumbel--Hougaard family \citep{Gum61,Hou86}. 

\subsection{Suggested data-adaptive smooth estimators and their finite-sample performance}
\label{sec:suggested}

The previous experiments confirm what the theoretical hints given in Section~\ref{sec:surv:cop} already suggested: it seems meaningful to take the survival copula $\bar \Cc^{\sss \Xc_{1:n}}$ of the smoothing distributions equal to the data generating copula $C$. Since the latter is unknown and since the empirical beta copula $C_{1:n}^\Beta$ in~\eqref{eq:C:1n:beta} is probably one of the best available estimators of~$C$, it is a natural choice for~$\bar \Cc^{\sss \Xc_{1:n}}$. As far as the dispersion parameter $\rho$ is concerned, the results reported in Section~\ref{sec:influence} suggest that taking its value to be larger than one, say $\rho \in \{2,4\}$, may be a good general choice. Two suggested smooth data-adaptive copula estimators are thus $C_{1:n}^{\sss{\bar \Bc, C_{1:n}^\Beta}}$ and $C_{1:n}^{\sss{\bar \beta, C_{1:n}^\Beta}}$ with $\rho \in \{2,4\}$. The first one, generically defined in~\eqref{eq:ec:BetaBin:marg}, has scaled beta-binomial smoothing survival margins and uses the empirical beta copula for $\bar \Cc^{\sss \Xc_{1:n}}$. The second one, generically defined in~\eqref{eq:ec:Beta:marg}, has beta smoothing survival margins and also uses the empirical beta copula for $\bar \Cc^{\sss \Xc_{1:n}}$. As competitors to these estimators in our Monte Carlo experiments, we considered:
\begin{itemize}
\item the empirical beta copula $C_{1:n}^\Beta$  (which is the same as the estimator $C_{1:n}^{\sss{\bar B, \Pi}}$ in~\eqref{eq:ec:Bin:marg} with $\bar \Cc^{\sss \Xc_{1:n}} = \Pi$),
\item the estimators $C_{1:n}^{\sss{\bar \Bc, \Pi}}$ and $C_{1:n}^{\sss{\bar \beta, \Pi}}$ with $\rho \in \{ 2,4\}$, which can be regarded as smoother versions of $C_{1:n}^\Beta = C_{1:n}^{\sss{\bar B, \Pi}}$,
\item and the estimators $C_{1:n}^{\sss{\bar \beta, \Pi}}$ and $C_{1:n}^{\sss{\bar \beta, C_{1:n}^\Beta}}$ with $\rho = 0.5$ which are marginally rougher than all of the previously considered ones.
\end{itemize}

\begin{figure}[t!]
\begin{center}
  \includegraphics*[width=1\linewidth]{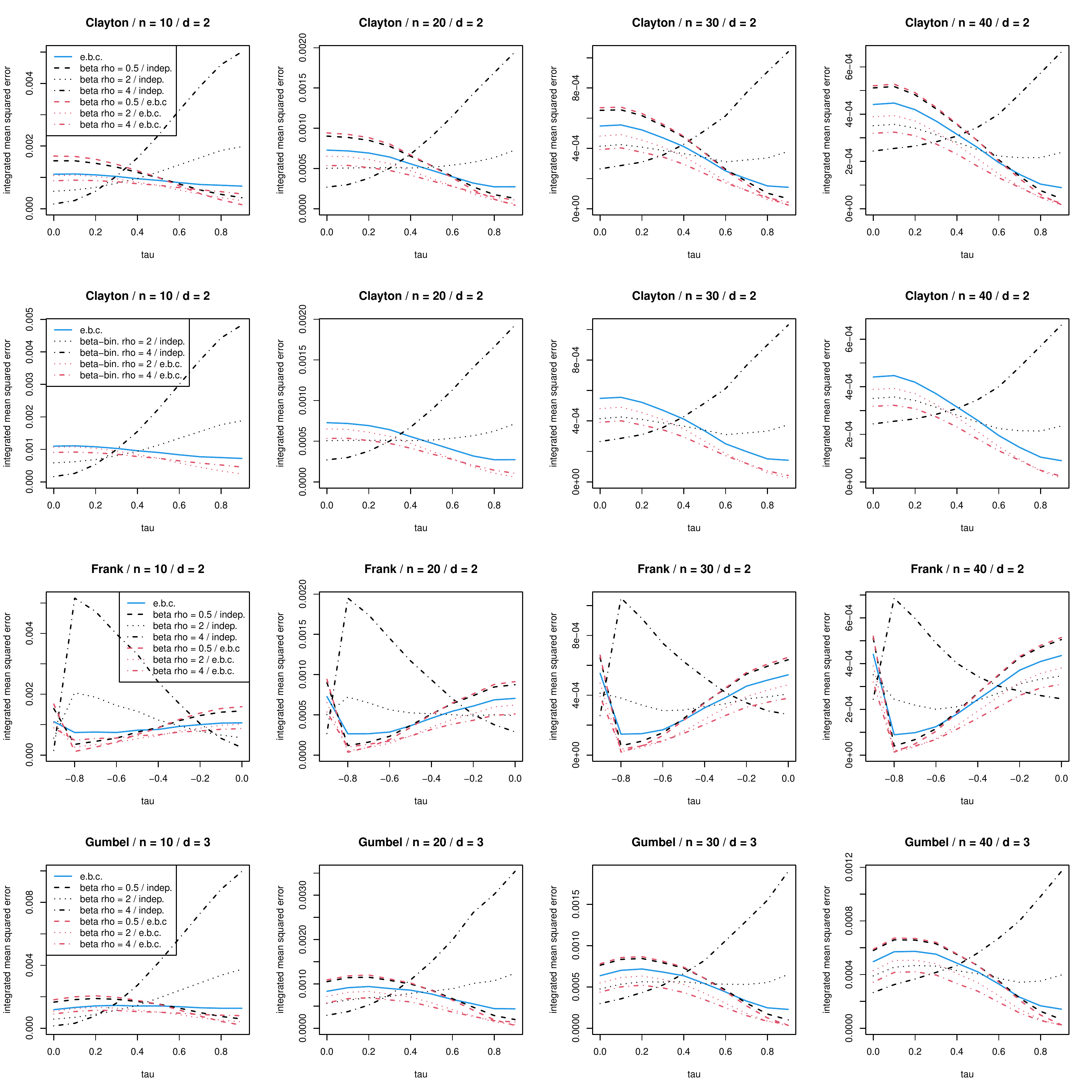}
  \caption{\label{fig:rho:tau} Estimated integrated mean squared errors against Kendall's tau $\tau$ for $\tau$ in either $\{0,0.1,\dots,0.9\}$ or $\{-0.9,-0.8,\dots,0\}$ of various estimators of $C$ for $C$ the bivariate Clayton (first and second row of graphs), the bivariate Frank (third row of graphs) or the trivariate Gumbel--Hougaard copula (fourth row of graphs) for $n \in \{10,20,30,40\}$. The abbreviations used in the legends are those defined in Section~\ref{sec:suggested}.}
\end{center}
\end{figure}

In a first series of experiments, we considered $d \in \{2,3\}$ and $C$ to be either the $d$-dimensional Clayton, Gumbel--Hougaard or Frank copula whose bivariate margins have a Kendall's tau of $\tau$ and we estimated the integrated mean squared error of the aforementioned estimators for $n \in \{10,20,30,40\}$ and for $\tau$ in either $\{0,0.1,\dots,0.9\}$ or in $\{-0.9,-0.8,\dots,0\}$. The estimated integrated mean squared errors are given against Kendall's tau in the panels of Fig.~\ref{fig:rho:tau} in which the following abbreviations are used:
\begin{itemize}
\item ``e.b.c.'' for ``empirical beta copula'',
\item ``beta-bin. rho = 2 / indep'' for the estimator $C_{1:n}^{\sss{\bar \Bc, \Pi}}$ with $\rho = 2$,
\item ``beta-bin. rho = 2 / e.b.c.'' for the estimator $C_{1:n}^{\sss{\bar \Bc, C_{1:n}^\Beta}}$ with $\rho = 2$,
\item ``beta rho = 2 / indep'' for the estimator $C_{1:n}^{\sss{\bar \beta, \Pi}}$ with $\rho = 2$,
\item ``beta rho = 2 / e.b.c.'' for the estimator $C_{1:n}^{\sss{\bar \beta, C_{1:n}^\Beta}}$ with $\rho = 2$,
\end{itemize}
and similarly for the estimators with $\rho = 4$. 

The following conclusions can be drawn from Fig.~\ref{fig:rho:tau}:
\begin{itemize}
\item the estimators $C_{1:n}^{\sss{\bar \Bc, C_{1:n}^\Beta}}$ and $C_{1:n}^{\sss{\bar \beta, C_{1:n}^\Beta}}$ with $\rho \in \{2,4\}$ are uniformly better than the empirical beta copula in terms of integrated mean squared error for the data generating models under consideration; setting $\rho$ to 4 is overall better than setting $\rho=2$, except when, approximately, $n \leq 30$ and $|\tau| \geq 0.6$,
\item as expected from Section~\ref{sec:surv:cop} and the experiments reported in Section~\ref{sec:influence}, all other settings being equal, the estimators for which $\bar \Cc^{\sss \Xc_{1:n}} = \Pi$ are better than their analogs for which $\bar \Cc^{\sss \Xc_{1:n}} = C_{1:n}^\Beta$ when $\tau$ is close to~0; their advantage, particularly large for small $n$, decreases as $n$ increases, that is, as the performance of the pilot estimator $C_{1:n}^\Beta$ improves,
\item as expected from the experiments reported in Section~\ref{sec:influence}, setting $\rho = 0.5$ in the expressions of $C_{1:n}^{\sss{\bar \beta, C_{1:n}^\Beta}}$ and $C_{1:n}^{\sss{\bar \beta, \Pi}}$ can have advantages only when $|\tau|$ is large and $n$ is small,
\item as can be seen from the first and second rows of plots (and from other non reported results), for the same value of $\rho > 1$, there is hardly any difference between the estimators $C_{1:n}^{\sss{\bar \Bc, C_{1:n}^\Beta}}$ and $C_{1:n}^{\sss{\bar \beta, C_{1:n}^\Beta}}$ (resp.\ $C_{1:n}^{\sss{\bar \Bc, \Pi}}$ and $C_{1:n}^{\sss{\bar \beta, \Pi}}$) in terms of integrated mean squared error.
\end{itemize}
Note that we have no explanations for the somehow surprising behavior occurring in the third line of graphs when $\tau$ changes from $-0.8$ to $-0.9$.

In a last series of experiments, we compared the integrated mean squared error of $C_{1:n}^\Beta$ with that of $C_{1:n}^{\sss{\bar \beta, C_{1:n}^\Beta}}$, $C_{1:n}^{\sss{\bar \beta, \Pi}}$ $C_{1:n}^{\sss{\bar \Bc, C_{1:n}^\Beta}}$ and $C_{1:n}^{\sss{\bar \Bc, \Pi}}$ for $\rho \in \{2,4\}$ for various bivariate and trivariate data-generating models and $n \in \{10,20,\dots,100\}$. Specifically, we considered:
\begin{itemize}
\item bivariate normal copulas with a Kendall's tau $\tau \in \{-0.75,-0.5,-0.25,0.25,0.5,0.75\}$,
\item non-exchangeable bivariate Khoudraji--Clayton copulas with first shape parameter $s_1 \in \{0.4, 0.6,0.8\}$, second shape parameter $s_2 = 0.95$ and Clayton copula parameter equal to 6 (see Section 3.4.2 in \cite{HofKojMaeYan18} and Fig.~3.19 therein in particular),
\item trivariate Gumbel--Hougaard copulas whose bivariate margins have a Kendall's tau $\tau \in \{0.25,0.5,0.75\}$,
\item the trivariate $t$ copula with 4 degrees of freedom and pairwise correlation parameters $\rho_{12} = -0.2$, $\rho_{13} = 0.5$ and $\rho_{23} = 0.4$ as in \cite{SegSibTsu17},
\item the trivariate nested Archimedean copula with Frank generators and with a Kendall's tau equal to 0.3 at the upper node and 0.6 at the lower node as in \cite{SegSibTsu17},
\item the trivariate nested Archimedean copula with Frank generators and with a Kendall's tau equal to 0.5 at the upper node and 0.8 at the lower node.
\end{itemize}

\begin{figure}[t!]
\begin{center}
  \includegraphics*[width=0.84\linewidth]{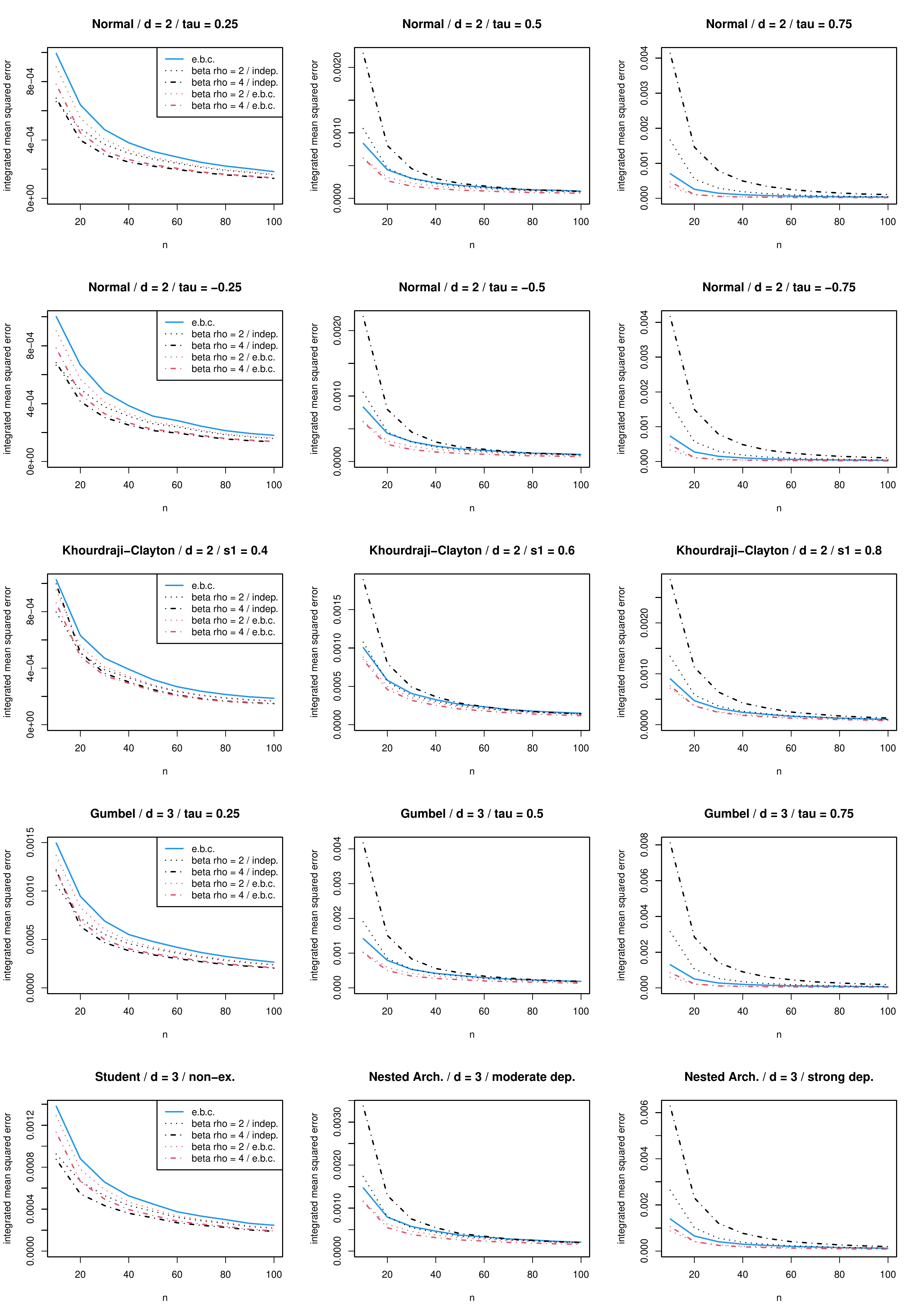}
  \caption{\label{fig:all:models} Estimated integrated mean squared errors against $n \in \{10,20, \dots, 100\}$ of five estimators of $C$ for the bivariate and trivariate data-generating models considered in Section~\ref{sec:suggested}.}
\end{center}
\end{figure}

\begin{figure}[t!]
\begin{center}
  \includegraphics*[width=0.83\linewidth]{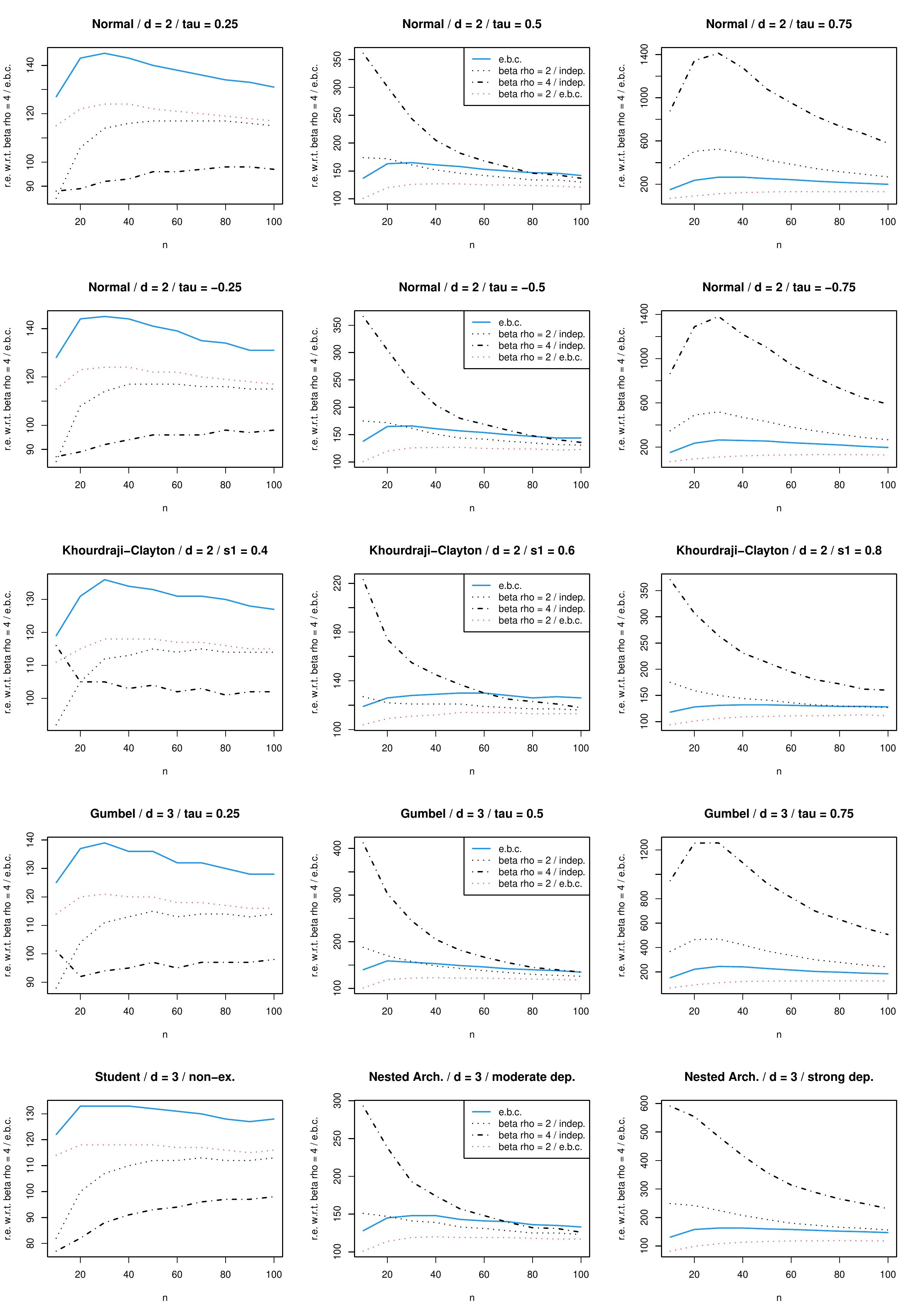}
  \caption{\label{fig:all:models:2} Estimated relative efficiencies (r.e.) with respect to the estimator $C_{1:n}^{\sss{\bar \beta, C_{1:n}^\Beta}}$ with $\rho = 4$ against $n \in \{10,20, \dots, 100\}$ for the four other estimators of $C$ appearing in Fig.~\ref{fig:all:models} for the bivariate and trivariate data-generating models considered in Section~\ref{sec:suggested}.}
\end{center}
\end{figure}

The estimated integrated mean squared errors are plotted against $n \in \{10,20, \dots, 100\}$ in the graphs of Fig.~\ref{fig:all:models} for the estimators (other than $C_{1:n}^\Beta$) that have beta smoothing survival margins. As already mentioned, the estimated integrated mean squared errors of the estimators having scaled beta-binomial smoothing survival margins (all other settings being equal) are almost identical. Again, as can be seen from Fig.~\ref{fig:all:models}, the estimators $C_{1:n}^{\sss{\bar \Bc, C_{1:n}^\Beta}}$ and $C_{1:n}^{\sss{\bar \beta, C_{1:n}^\Beta}}$ with $\rho \in \{2,4\}$ are uniformly better than the empirical beta copula in terms integrated mean squared error for all the data generating models under consideration, the best of the two being overall the one with $\rho = 4$, except for very small $n$. For an easier evaluation of the possible gain resulting from the use of the latter estimator, Fig.~\ref{fig:all:models:2} reports estimated relative efficiencies with respect to the estimator $C_{1:n}^{\sss{\bar \beta, C_{1:n}^\Beta}}$ with $\rho = 4$ against $n \in \{10,20, \dots, 100\}$ (the relative efficiency of an estimator with respect to $C_{1:n}^{\sss{\bar \beta, C_{1:n}^\Beta}}$ with $\rho = 4$ is simply the ratio of its integrated mean squared error to that of $C_{1:n}^{\sss{\bar \beta, C_{1:n}^\Beta}}$ with $\rho = 4$ multiplied by 100). As already observed in the previous set of experiments, the estimators $C_{1:n}^{\sss{\bar \Bc, \Pi}}$ and $C_{1:n}^{\sss{\bar \beta, \Pi}}$ with $\rho \in \{2,4\}$ can be of interest only when the true copula is close to the independence copula $\Pi$.


\section{Asymptotics of related sequential empirical processes}
\label{sec:asym}

Recall the definition given in~\eqref{eq:C:1n:nu} of the smooth, possibly data-adaptive empirical copulas $C_{1:n}^\nu$ investigated in this work. Conditions under which empirical processes of the form $\sqrt{n}(C_{1:n}^\nu - C)$ converge weakly will be a consequence of the sequential results to be stated in this section. The reason for considering this more general setting follows from our intention to apply the studied smooth estimators in change-point analysis in a forthcoming project.

Before providing the asymptotic theory for the related sequential empirical processes, we start by briefly explaining how the quantities introduced in Section~\ref{sec:class} can be defined from substretches of the available stretch of observations $\Xc_{1:n} = (\bm X_1,\dots,\bm X_n)$. Given a substretch $\Xc_{k:l} = (\bm X_k,\dots,\bm X_l)$, $1 \leq k \leq l \leq n$, let $F_{k:l,j}$ be the empirical d.f.\ computed from the $j$th component subsample $X_{kj},\dots,X_{lj}$, $j \in \{1,\dots,d\}$ of $\Xc_{k:l}$. Then, let $R_{ij}^{k:l} = (l-k+1) F_{k:l,j}(X_{ij})$ be the (maximal) rank of $X_{ij}$ among $X_{kj},\dots,X_{lj}$. Furthermore, let $\bm R^{k:l}_i =  \left(R_{i1}^{k:l}, \dots, R_{id}^{k:l} \right)$ and $\pobs{U}^{k:l}_i = \bm R^{k:l}_i / (l-k+1)$, $i \in \{k,\dots,l\}$, be the multivariate ranks and the multivariate scaled ranks, respectively, obtained from $\Xc_{k:l}$. The empirical copula $C_{k:l}$ of $\Xc_{k:l}$ at $\bm u = (u_1,\dots,u_d) \in [0,1]^d$ is then
\begin{equation}
\label{eq:C:kl}
C_{k:l}(\bm u) = \frac{1}{l-k+1} \sum_{i=k}^l  \prod_{j=1}^d \1\left( \frac{R_{ij}^{k:l}}{l-k+1} \leq u_j \right)
= \frac{1}{l-k+1} \sum_{i=k}^l  \1(\pobs{U}^{k:l}_i \leq \bm u),
\end{equation}
and, by analogy with~\eqref{eq:C:1n:nu}, the corresponding smooth version $C_{k:l}^\nu$ based on the possibly random smoothing distributions $\nu_{\bm u}^{\sss \Xc_{k:l}}$, $\bm u \in [0,1]^d$, is defined by
\begin{equation}
  \label{eq:C:kl:nu}
  C_{k:l}^\nu(\bm u) = \int_{[0,1]^d} C_{k:l}(\bm w) \dd \nu_{\bm u}^{\sss \Xc_{k:l}}(\bm w), 
  \qquad \bm u \in [0,1]^d.
\end{equation}
Note that, if $k > l$, we adopt the convention that $C_{k:l} =  C_{k:l}^\nu = 0$ and that, for any $\bm u \in [0,1]^d$, $\nu_{\bm u}^{\sss \Xc_{k:l}}$ is the Dirac measure at $\bm u$.

We can now define the corresponding sequential empirical copula processes which are the main focus of this section. Let $\Delta = \{(s,t) \in [0,1]^2 : s \leq t \}$ and let $\lambda_n(s,t) = (\ip{nt} - \ip{ns})/n$, $(s, t) \in \Delta$. The two-sided sequential empirical copula process was defined in \cite{BucKoj16} by
\begin{equation}
  \label{eq:Cb:n}
  \Cb_n(s,t,\bm u) = \sqrt{n}\lambda_n(s,t)\{C_{\ip{ns} +1 :\ip{nt}}(\bm u)-C(\bm u)\}, \qquad (s,t) \in \Delta, \bm u \in [0,1]^d,
\end{equation}
where $C_{\ip{ns} +1 :\ip{nt}}$ is generically defined in~\eqref{eq:C:kl}. The two-sided sequential empirical copula process corresponding to the smooth empirical copula in~\eqref{eq:C:kl:nu} is then naturally defined by
\begin{equation}
\label{eq:Cb:n:nu}
\Cb_n^\nu(s,t,\bm u)=\sqrt{n}\lambda_n(s,t)\{C_{\ip{ns} +1 :\ip{nt} }^\nu(\bm u)-C(\bm u)\}, \qquad (s,t) \in \Delta, \bm u \in [0,1]^d.
\end{equation}
The aim of this section is to establish the asymptotics of the latter.

As frequently done in the literature, we consider the following nonrestrictive condition on the unknown copula $C$ initially proposed in \cite{Seg12}.

\begin{cond}[Smooth partial derivatives]
\label{cond:pd}
For any $j \in \{1,\dots,d\}$, the partial derivative $\dot C_j = \partial C/\partial u_j$ exists and is continuous on the set $\{ \bm u \in [0, 1]^d : u_j \in (0,1) \}$.
\end{cond}

The latter condition can be considered nonrestrictive because, as shall become clear in the forthcoming developments, it is necessary for the candidate weak limit of both $\Cb_n$ and $\Cb_n^\nu$ to exist pointwise and have continuous trajectories. In the rest of the paper, following~\cite{BucVol13}, for any $j \in \{1,\dots,d\}$, $\dot C_j$ is arbitrarily defined to be zero on the set $\{ \bm{u} \in [0, 1]^d : u_j \in \{0,1\} \}$, which implies that, under Condition~\ref{cond:pd}, $\dot C_j$ is defined on the whole of $[0, 1]^d$.

In view of the mathematical derivations carried out in \cite{SegSibTsu17} to study the asymptotics of the empirical beta copula process and with~\eqref{eq:var:W} in mind, it seems particularly meaningful to impose the following condition on the smoothing distributions involved in the definition of the smooth empirical copulas under consideration.

\begin{cond}[Variance condition]
  \label{cond:var:W}
  There exists a constant $\kappa > 0$ such that, for any $n \in \N$, $\bm x \in (\R^d)^n$, $\bm u \in [0,1]^d$ and $j \in \{1,\dots,d\}$, $\Var( W_{j,u_j}^{\bm x}) \leq \kappa u_j(1-u_j) / n$. 
\end{cond}

\begin{remark}
As verified at the beginning of Section~\ref{sec:surv:marg:dis}, the previous condition is clearly satisfied if, for any $n \in \N$, $\bm x \in (\R^d)^n$, $\bm u \in [0,1]^d$ and $j \in \{1,\dots,d\}$, $W_{j,u_j}^{\bm x} = S_{n,j,u_j} / n$, where $S_{n,j,u_j}$ is Binomial$(n,u_j)$. From~\eqref{eq:beta-binomial:var:2} and~\eqref{eq:beta:var:2}, we see that it will also be satisfied for smoothing distributions with scaled beta-binomial and beta survival margins as defined in Sections~\ref{sec:surv:marg:dis} and~\ref{sec:beta:margin} provided that the dispersion parameter $\rho$ is constant. For versions of these models in which $\rho$ is data-adaptive, one needs to assume that there exists a constant $\gamma > 0$, such for any $n \in \N$ and $\bm x \in (\R^d)^n$, the dispersion parameter $\rho = \rho^{\bm x}$ (which may thus depend on $\bm x$) is bounded by $\gamma$.
\end{remark}

In the rest of this section, unless stated otherwise, all convergences are with respect to $n \to \infty$. Furthermore, the arrow~`$\leadsto$' denotes weak convergence in the sense of Definition~1.3.3 in \cite{vanWel96} and, given a set $T$, $\ell^\infty(T)$ (resp.\ $\Cc(T)$) represents the space of all bounded (resp.\ continuous) real-valued functions on $T$ equipped with the uniform metric. The following result is proven in Appendix~\ref{proof:thm:Cb:n:nu}.

\begin{thm}[Asymptotics of $\Cb_n^\nu$]
\label{thm:Cb:n:nu}
Assume that Conditions~\ref{cond:pd} and~\ref{cond:var:W} hold, and that $\Cb_n \leadsto \Cb_C$ in $ \ell^\infty(\Delta \times [0,1]^d)$, where the trajectories of the limiting process $\Cb_C$ are continuous almost surely. Then,
$$
\sup_{(s,t) \in \Delta \atop \bm u \in [0,1]^d} |\Cb_n^\nu(s,t,\bm u) - \Cb_n(s,t,\bm u) | = o_\Pr(1).
$$
Consequently,  $\Cb_n^\nu \leadsto \Cb_C$ in $\ell^\infty(\Delta \times[0,1]^d)$.
\end{thm}

Roughly speaking, the smooth sequential empirical copula process $\Cb_n^\nu$ in~\eqref{eq:Cb:n:nu} is asymptotically equivalent to the classical sequential empirical copula process $\Cb_n$ in~\eqref{eq:Cb:n} when the latter converges weakly to a limiting process whose trajectories are continuous almost surely. From Section~3 in \cite{BucKoj16}, a key assumption for such a convergence to hold concerns the weak convergence of the two-sided sequential empirical process
\begin{equation}
\label{eq:Bb:n}
  \Bb_n(s, t, \bm{u}) = \frac{1}{\sqrt{n}} \sum_{i=\ip{ns}+1}^{\ip{nt}} \{\1(\bm U_i \leq \bm{u}) - C(\bm{u}) \}, \qquad (s, t,\bm{u}) \in \Delta \times [0, 1]^d,
\end{equation}
where $\bm U_1,\dots,\bm U_n$ is the unobservable sample obtained from $\Xc_{1:n} = (\bm X_1, \dots, \bm X_n)$ by the probability integral transformations $U_{ij} = F_j(X_{ij})$, $i \in \{1,\dots,n\}$, $j \in \{1,\dots,d\}$, and with the convention that $\Bb_n(s, t, \cdot) = 0$ if $\ip{nt} - \ip{ns} = 0$. Specifically, the following condition was considered in \cite{BucKoj16}.

\begin{cond}[Weak convergence of $\Bb_n(0, \cdot, \cdot)$]
\label{cond:Bn}
The sequential empirical process $\Bb_n(0, \cdot, \cdot)$ converges weakly in $\ell^\infty([0,1]^{d+1})$ to a tight centered Gaussian process $\Zb_C$ concentrated on
\begin{align*}
\{ f \in \Cc([0,1]^{d+1}) : f(s,\bm u) = 0 \text{ if one of the components of $(s,\bm u)$ is 0, and } f(s,1,\dots,1) = 0 \text{ for all } s \in (0,1] \}.
\end{align*}
\end{cond}

Under Condition~\ref{cond:Bn}, it immediately follows from the continuous mapping theorem that $\Bb_n \leadsto \Bb_C$ in $\ell^\infty(\Delta \times [0, 1]^d)$, where
\begin{equation}
\label{eq:Bb:C}
  \Bb_C(s, t,\bm{u}) = \Z_C(t,\bm{u}) - \Z_C(s,\bm{u}),   \qquad (s, t,\bm{u}) \in \Delta \times [0,1]^d.
\end{equation}

When the underlying stationary time series $(\bm X_i)_{i \in \Z}$ consists of independent random vectors, Condition~\ref{cond:Bn} is a direct consequence, for instance, of Theorem 2.12.1 in \cite{vanWel96}. More generally, Condition~\ref{cond:Bn} also holds when $(\bm X_i)_{i \in \Z}$ is strongly mixing. Given a stationary time series $(\bm Y_i)_{i \in \Z}$, denote by $\Fc_j^k$ the $\sigma$-field generated by $(\bm Y_i)_{j \leq i \leq k}$, $j, k \in \Z \cup \{-\infty,+\infty \}$, and recall that the strong mixing coefficients corresponding to the stationary sequence $(\bm Y_i)_{i \in \Z}$ are then defined by
\begin{equation*}
\alpha_r^{\bm Y} = \sup_{A \in \Fc_{-\infty}^0,B\in \Fc_{r}^{+\infty}} \big| \Pr(A \cap B) - \Pr(A) \Pr(B) \big|, \qquad r \in \N, \, r > 0,
\end{equation*}
with  $\alpha_0^{\bm Y} = 1/2$ and that the sequence $(\bm Y_i)_{i \in \Z}$ is said to be strongly mixing if $\alpha_r^{\bm Y} \to 0$ as $r \to \infty$. Then, if the stretch $\bm U_1,\dots,\bm U_n$ is drawn from a time series $(\bm U_i)_{i \in \Z}$ whose strong mixing coefficients satisfy $\alpha_r^{\bm U} = O(r^{-a})$ with $a > 1$ (which occurs if the strong mixing coefficients of the time series $(\bm X_i)_{i \in \Z}$ satisfy $\alpha_r^{\bm X} = O(r^{-a})$ with $a > 1$), Theorem~1 in \cite{Buc15} implies that Condition~\ref{cond:Bn} holds with the covariance function of the process $\Zb_C$ being
$$
\Cov\{\Zb_C(s,\bm{u}), \Zb_C(t,\bm v)\} = \min(s,t) \sum_{k \in \Z} \Cov\{ \1(\bm U_0 \leq \bm{u}), \1(\bm U_k \leq \bm v) \}.
$$

Recall that, for any $j \in \{1,\dots,d\}$ and any $\bm{u} \in [0, 1]^d$, $\bm{u}^{(j)}$ is the vector of $[0, 1]^d$ defined by $u^{(j)}_i = u_j$ if $i = j$ and 1 otherwise. The following result is then an immediate consequence of Theorem~3.4 in \cite{BucKoj16} and Proposition~3.3 of \cite{BucKojRohSeg14}, and can be used to obtain a straightforward corollary of Theorem~\ref{thm:Cb:n:nu}.

\begin{thm}[Asymptotics of $\Cb_n$]
Under Conditions~\ref{cond:pd} and~\ref{cond:Bn},
\begin{equation*}
  \sup_{(s, t,\bm{u})\in \Delta \times [0, 1]^d} \left| \Cb_n(s, t, \bm{u}) - \tilde \Cb_n(s, t, \bm{u}) \right| = o_\Pr(1),
\end{equation*}
where
\begin{equation*}
\tilde \Cb_n(s, t, \bm{u}) = \Bb_n(s, t, \bm{u}) - \sum_{j=1}^d \dot C_j(\bm{u}) \, \Bb_n(s, t, \bm{u}^{(j)}), \qquad (s, t, \bm{u}) \in \Delta \times [0, 1]^d,
\end{equation*}
and $\Bb_n$ is defined in~\eqref{eq:Bb:n}. Consequently, $\Cb_n \leadsto \Cb_C$ in $\ell^\infty(\Delta \times [0, 1]^d)$, where
\begin{equation*}
\Cb_C(s, t, \bm{u})  = \Bb_C(s, t, \bm{u}) - \sum_{j=1}^d \dot C_j(\bm{u}) \, \Bb_C(s, t, \bm{u}^{(j)}), \qquad (s, t, \bm{u}) \in \Delta \times [0, 1]^d,
\end{equation*}
and $\Bb_C$ is defined in~\eqref{eq:Bb:C}.
\end{thm}

\section{Concluding remarks}

A broad class of smooth, possibly data-adaptive empirical copulas containing the empirical beta copula was studied. A good general choice within this class appears to be the estimator $C_{1:n}^{\sss{\bar \Bc, C_{1:n}^\Beta}}$ generically defined in~\eqref{eq:ec:BetaBin:marg} that uses the empirical beta copula as survival copula of the smoothing distribution and scaled beta-binomial smoothing survival margins with dispersion parameter $\rho = 4$. From Corollary~\ref{cor:copula}, the latter is a genuine copula \ik{in the absence of ties in the component samples}. In our bivariate and trivariate Monte Carlo experiments, it was found to be uniformly better than the empirical beta copula for all the considered data-generating models. However, its better finite-sample performance compared to the empirical beta copula comes at the price of a higher computational cost since this more complex estimator uses the empirical beta copula as a pilot estimator of the true unknown copula. Specifically, we see from~\eqref{eq:ec:BetaBin:marg} that one estimation of the unknown copula at a given point $\bm u \in [0,1]^d$ will require $n$ evaluations of the empirical beta copula. Whether this overhead is acceptable for a given application can be assessed using the \textsf{R} implementation of the estimator available on the web page of the first author.

In addition to these finite-sample results, conditions under which sequential empirical copula processes constructed from the studied general class of smooth, possibly data-adaptive estimators converge weakly were provided. In a forthcoming project, these results will be used to show the asymptotic validity of a sequential dependent multiplier bootstrap \cite{BucKoj16,BucKojRohSeg14} allowing to apply the proposed smooth, possibly data-adaptive empirical copulas to change-point detection in a time series setting. 


\setcounter{section}{0}
\renewcommand{\thesection}{\Alph{section}}

\section{Proofs of Lemma~\ref{lem:beta:mixture}, Proposition~\ref{prop:unif:margin}, Proposition~\ref{prop:mult:df} and Proposition~\ref{prop:mult:df:cont}}
\label{proofs:short}

\begin{proof}[\bf Proof of Lemma~\ref{lem:beta:mixture}]
For any $\bm u \in [0,1]^d$, $n \in \N$ and $r_1,\dots,r_d \in \{1,\dots,n\}$, we have
\begin{align*}
  \int_{[0,1]^d} &\prod_{j=1}^d \1(r_j/n \leq w_j) \dd \mu_{n,\bm u}(\bm w) = \Ex \left\{ \prod_{j=1}^d \1(r_j/n \leq S_{n,j,u_j}/n) \right\} = \prod_{j=1}^d \Ex \{  \1(r_j \leq S_{n,j,u_j}) \} = \prod_{j=1}^d \Pr (S_{n,j,u_j} \geq r_j) \\
  &= \prod_{j=1}^d \{ 1 - \Pr (S_{n,j,u_j} \leq r_j-1) \} = \prod_{j=1}^d \left\{ 1 - \sum_{s=0}^{r_j-1} \binom{n}{s} u_j^s (1 - u_j)^{n-s} \right\} \\
                                                                           &= \prod_{j=1}^d \left\{ \sum_{s=0}^n \binom{n}{s} u_j^s (1 - u_j)^{n-s} - \sum_{s=0}^{r_j-1} \binom{n}{s} u_j^s (1 - u_j)^{n-s} \right\} = \prod_{j=1}^d \sum_{s=r_j}^n \binom{n}{s} u_j^s (1 - u_j)^{n-s} = \prod_{j=1}^d \Beta_{n,r_j}(u_j).
\end{align*}
\end{proof}

\begin{proof}[\bf Proof of Proposition~\ref{prop:unif:margin}]
Fix $j \in \{1,\dots,d\}$ and $\bm u \in [0,1]^d$. Then, from~\eqref{eq:var:W}, all the components of $\bm W_{\bm u^{(j)}}^{\sss\Xc_{1:n}}$ except possibly the $j$th are deterministic and equal to $1$. It follows from~\eqref{eq:K} and~\eqref{eq:C:1n:K} that, almost surely,
\begin{align*}
  C_{1:n}^\nu(\bm u^{(j)}) &= \frac{1}{n} \sum_{i=1}^n \Kc_{\sss{\bm R^{1:n}_i}}^{\sss\Xc_{1:n}}(\bm u^{(j)})
                             = \frac{1}{n} \sum_{i=1}^n \int_{[0,1]^d} \1(R^{1:n}_{ij} / n \leq w_j) \dd\nu_{\bm u}^{\sss \Xc_{1:n}}(\bm w) = \frac{1}{n} \sum_{i=1}^n \int_{[0,1]^d} \1(i / n \leq w_j) \dd\nu_{\bm u}^{\sss \Xc_{1:n}}(\bm w) \\
                           &= \frac{1}{n} \sum_{i=1}^{n} \Ex \left\{ \1 \left( \frac{i}{n} \leq W_{j,u_j}^{\sss \Xc_{1:n}} \right) \mid \Xc_{1:n} \right\} = \Ex \left\{ \frac{1}{n} \sum_{i=1}^{n}  \1 \left( i \leq n W_{j,u_j}^{\sss \Xc_{1:n}} \right) \mid \Xc_{1:n} \right\} = \Ex \left(  \frac{\ip{n W_{j,u_j}^{\sss \Xc_{1:n}}}}{n} \mid \Xc_{1:n} \right).
\end{align*}
\end{proof}

\begin{proof}[\bf Proofs of Propositions~\ref{prop:mult:df} and~\ref{prop:mult:df:cont}]
  We only provide the proof of Proposition~\ref{prop:mult:df} as the proof of Proposition~\ref{prop:mult:df:cont} is very similar.

  Let us check that $C_{1:n}^\nu$ in~\eqref{eq:C:1n:dis} satisfies the four properties listed in Theorem~1.2.11 of \cite{DurSem15} necessary for it to be a multivariate d.f. From~\eqref{eq:K:dis}, under Condition~\ref{cond:smooth:cop}, for any $\bm u \in [0,1]^d$ and $i \in \{1,\dots,n\}$, one has  that, almost surely
$$
\Kc_{\bm R_i^{1:n}}^{\sss \Xc_{1:n}}(\bm u) = \bar \Cc^{\sss \Xc_{1:n}} [ \bar \Fc_{1,u_1}^{\sss \Xc_{1:n}}\{ (R_{i1}^{1:n} - 1) / n\}, \dots, \bar \Fc_{d,u_d}^{\sss \Xc_{1:n}} \{ (R_{id}^{1:n} - 1) / n \} ].
$$
Since, for any $\bm u \in [0,1]^d$ and $j \in \{1,\dots,d\}$, $W_{j,u_j}^{\sss \Xc_{1:n}} = 0$ if the $j$th coordinate of $\bm u$ is zero and~$\bar \Cc^{\sss \Xc_{1:n}}$ is a copula, we obtain that, with  probability 1, $\Kc_{\bm R_i^{1:n}}^{\sss \Xc_{1:n}}(\bm u) = 0$ for all $i \in \{1,\dots,n\}$ if at least one coordinate of $\bm u$ is zero. The latter implies that the estimator $C_{1:n}^\nu$ in~\eqref{eq:C:1n:dis} is grounded. Similarly, if $\bm u = \bm 1 \in [0,1]^d$, with  probability 1 $W_{j,u_j}^{\sss \Xc_{1:n}}  = 1$ for all $j \in \{1,\dots,d\}$ and $\Kc_{\bm R_i^{1:n}}^{\sss \Xc_{1:n}}(\bm u) = 1$ for all $i \in \{1,\dots,n\}$, which implies that $C_{1:n}^\nu(\bm 1) = 1$ almost surely.

From Condition~\ref{cond:smooth:surv:marg} and the fact that~$\bar \Cc^{\bm x}$ is a copula for all $\bm x\in (\R^d)^n$, we have that, for any $\bm x \in (\R^d)^n$, $j \in \{1,\dots,d\}$, $\bm u \in [0,1]^d$ and $\bm w \in [0,1)^d$, the function from $[0,1]$ to $[0,1]$ defined by
$$
t \mapsto \bar \Cc^{\bm x} \{ \bar \Fc_{1,u_1}^{\bm x}(w_1), \dots, \bar \Fc_{j,u_{j-1}}^{\bm x}(w_{j-1}), \bar \Fc_{j,t}^{\bm x}(w_j), \bar \Fc_{j,u_{j+1}}^{\bm x}(w_{j+1}), \dots, \bar \Fc_{d,u_d}^{\bm x} (w_d) \}
$$
is right-continuous. The latter implies that, with  probability 1, for any $j \in \{1,\dots,d\}$ and $\bm u \in [0,1]^d$, the function from $[0,1]$ to $[0,1]$ defined by $t \mapsto C_{1:n}^\nu(u_1,\dots,u_{j-1},t,u_{j+1},\dots,u_d)$ is right-continuous as well.

For any $\bm a,\bm b\in[0,1]^d$ such that $\bm a\le\bm b$, let $(\bm a,\bm b] = (a_1,b_1] \times \dots \times (a_d,b_d]$ and, for any function $H$ on $[0,1]^d$, let the $H$-volume of $(\bm a,\bm b]$ be defined by
\begin{equation*}
  \Delta_{(\bm a,\bm b]}H = \sum_{\bm{i}\in\{0,1\}^d}(-1)^{\sum_{j=1}^d i_j} H \bigl( a_1^{i_1}b_1^{1-i_1},\dots,a_d^{i_d}b_d^{1-i_d} \bigr).
\end{equation*}
Then, some thought reveals that, for any $\bm a,\bm b\in[0,1]^d$ such that $\bm a \leq \bm b$, Condition~\ref{cond:smooth:surv:marg} and the fact that~$\bar \Cc^{\bm x}$ is $d$-increasing \cite[see, e.g.,][Definition 1.2.9]{DurSem15} for all $\bm x \in (\R^d)^n$ imply that, with  probability 1, $\Delta_{(\bm a,\bm b]} \Kc_{\bm R_i^{1:n}}^{\sss \Xc_{1:n}} \geq 0$ for all $i \in \{1,\dots,n\}$, which in turn implies that $\Delta_{(\bm a,\bm b]}C_{1:n}^\nu = \frac{1}{n} \sum_{i=1}^n \Delta_{(\bm a,\bm b]} \Kc_{\bm R_i^{1:n}}^{\sss \Xc_{1:n}} \geq 0$ almost surely, and thus that the estimator $C_{1:n}^\nu$ is $d$-increasing with  probability 1.

The desired conclusion finally follows from Theorem~1.2.11 in \cite{DurSem15}.
\end{proof}


\section{Proofs of Propositions~\ref{prop:Bin:cond}, ~\ref{prop:BetaBin:cond} and~\ref{prop:Beta:cond}}
\label{proofs:cond:marg}

The proofs of Propositions~\ref{prop:Bin:cond},~\ref{prop:BetaBin:cond} and~\ref{prop:Beta:cond} are based on six lemmas which we prove first.

\begin{lem}
  \label{lem:Bin:cont}
 For any $n \in \N$ and $w \in [0,n)$, the function $t \mapsto \bar B_{n,t}(w)$ is continuous on $[0,1]$.
\end{lem}

\begin{proof}[\bf Proof]
  Fix $n \in \N$ and $w \in [0,n)$. The desired result is an immediate consequence of the fact that 
  $$
  \bar B_{n,t}(w)  = 1 - \sum_{k=0}^{\ip{w}} {n \choose k} t^k (1-t)^{n-k}, \qquad t \in [0,1],
  $$
  is a polynomial in $t$.
\end{proof}
\begin{lem}
  \label{lem:Beta:cont}
 For any $n \in \N$, $\rho \in (0,n)$ and $w \in (0,1)$, the function $t \mapsto \bar \beta_{n,t,\rho}(w)$ is continuous on $[0,1]$.
\end{lem}

\begin{proof}[\bf Proof]
  Fix $n \in \N$, $\rho \in (0,n)$ and $w \in (0,1)$ and let $\tau = (n-\rho)/\rho > 0$. From the properties of the beta distribution Beta$(t\tau,(1-t)\tau)$, for any $t \in (0,1)$,
  $$
  \bar \beta_{n,t,\rho}(w) =  1 - \int_0^w f_{t,\tau}(x) \dd x,
  $$
where
$$
f_{t,\tau}(x) = \frac{\Gamma(\tau)}{\Gamma(t\tau)\Gamma\left\{(1-t)\tau\right\}} x^{t\tau - 1} (1-x)^{(1-t)\tau - 1}, \qquad x \in (0,1),
$$
and $\Gamma$ is the gamma function. Fix $t_0 \in (0,1)$ and let us verify that $t \mapsto \bar \beta_{n,t,\rho}(w)$ is continuous at $t_0$. To do so, let $t_m$ be an arbitrary sequence in $(0,1)$ such that $t_m \to t_0$ as $m \to \infty$. By continuity of $\Gamma$ on $(0,\infty)$, we have that, for any $x \in (0,1)$, $f_{t_m,\tau}(x) \to f_{t_0,\tau}(x)$ as $m \to \infty$. Furthermore, since there exists a constant $K > 0$ such that $\Gamma(z) \geq K$ for all $z \in (0,\infty)$, we have that, for any $m \in \N$ and $x \in (0,1)$, $f_{t_m,\tau}(x) \leq \Gamma(\tau)/K^2$. The dominated convergence theorem then implies that $\bar \beta_{n,t_m,\rho}(w) \to \bar \beta_{n,t_0,\rho}(w)$ as $m \to \infty$, which implies that the function $t \mapsto \bar \beta_{n,t,\rho}(w)$ is continuous on $(0,1)$. Furthermore, from Markov's inequality, we have that, for any $t \in (0,1)$, $\bar \beta_{n,t,\rho}(w) \leq t/w$, which combined with the fact that $\bar \beta_{n,0,\rho}(w) = 0$, implies that $\lim_{t \to 0^+} \bar \beta_{n,t,\rho}(w) = \bar \beta_{n,0,\rho}(w)$, and therefore that the function $t \mapsto \bar \beta_{n,t,\rho}(w)$ is continuous at $0$. Finally, using symmetry properties of the density of the Beta$(t\tau,(1-t)\tau)$, some thought reveals that, for any $t \in (0,1)$, $\bar \beta_{n,1-t,\rho}(w) =  \beta_{n,t,\rho}(1-w) = 1 - \bar \beta_{n,t,\rho}(1-w)$, which implies that $\lim_{t \to 1^-} \bar \beta_{n,t,\rho}(w) = \lim_{t \to 0^+} \bar \beta_{n,1-t,\rho}(w) = 1 - \lim_{t \to 0^+} \bar \beta_{n,t,\rho}(1-w) = 1 = \bar \beta_{n,1,\rho}(w)$ since, from Markov's inequality,  for any $t \in (0,1)$, $\bar \beta_{n,t,\rho}(1-w) \leq t/(1-w)$.
\end{proof}

\begin{lem}
  \label{lem:BetaBin:cont}
 For any $n \in \N$, $\rho \in (0,n)$ and $w \in [0,n)$, the function $t \mapsto \bar \Bc_{n,t,\rho}(w)$ is continuous on $[0,1]$.
\end{lem}

\begin{proof}[\bf Proof]
  Fix $n \in \N$, $\rho \in (1,n)$ and $w \in [0,n)$ and let $\tau = (n-\rho)/(\rho - 1) > 0$. From the properties of the beta-binomial distribution Beta-Binomial($n,t \tau, (1-t)\tau)$, for any $t \in (0,1)$, we have that
  $$
  \bar \Bc_{n,t,\rho}(w) = 1 - \sum_{k=0}^{\ip{w}} {n \choose k} \frac{\Gamma(k+t\tau) \Gamma\{n-k+(1-t)\tau\}\Gamma(\tau)}{\Gamma(n+\tau)\Gamma(t\tau)\Gamma\{(1-t)\tau\}}.
  $$
  Since $\Gamma$ is continuous on $(0,\infty)$, we immediately obtain that $t \mapsto \bar \Bc_{n,t,\rho}(w)$ is continuous on $(0,1)$. For some $t \in (0,1)$, let $X$ be Beta-Binomial($n,t \tau, (1-t)\tau)$. Using Markov's inequality, we have that
$$
\bar \Bc_{n,t,\rho}(w) = \Pr(X > w) = \Pr(X \geq \ip{w}+1) \leq \frac{nt}{\ip{w}+1}.
$$
The latter display implies that $\lim_{t \to 0^+} \bar \Bc_{n,t,\rho}(w) = 0 = \bar \Bc_{n,0,\rho}(w)$ since $X$ is degenerate and equal to 0 if $t = 0$. It remains to verify that $t \mapsto \bar \Bc_{n,t,\rho}(w)$ is continuous at $1$. Using symmetry properties of the probability mass function of the Beta-Binomial($n,t \tau, (1-t)\tau)$, some thought reveals that, for any $t \in (0,1)$, $\bar \Bc_{n,1-t,\rho}(w) =  \Bc_{n,t,\rho}(n-w) = 1 - \bar \Bc_{n,t,\rho}(n-w)$, which implies that $\lim_{t \to 1^-} \bar \Bc_{n,t,\rho}(w) = \lim_{t \to 0^+} \bar \Bc_{n,1-t,\rho}(w) = 1 - \lim_{t \to 0^+} \bar \Bc_{n,t,\rho}(n-w) = 1 = \bar \Bc_{n,1,\rho}(w)$ since, from Markov's inequality,  for any $t \in (0,1)$, $\bar \Bc_{n,t,\rho}(n-w) \leq nt/(\ip{n-w}+1)$.
\end{proof}

\begin{lem}
  \label{lem:Bin:increasing}
 For any $n \in \N$ and $w \in [0,n)$, the function $t \mapsto \bar B_{n,t}(w)$ is increasing on $[0,1]$.
\end{lem}

\begin{proof}[\bf Proof]
  Fix $n \in \N$. We need to prove that, for any $w \in [0,n)$, $\bar B_{n,t_1}(w) \leq \bar B_{n,t_2}(w)$ whenever $0 \leq t_1 \leq t_2 \leq 1$. Notice that, for any $w \in [0,n)$, $\bar B_{n,0}(w) = 0$ and $\bar B_{n,1}(w) = 1$ so that it suffices to prove that, for any $w \in [0,n)$, $\bar B_{n,t_1}(w) \leq \bar B_{n,t_2}(w)$ whenever $0 < t_1 \leq t_2 < 1$. Fix $0 < t_1 \leq t_2 < 1$ and let $X$ be Binomial$(n,t_1)$ and $Y$ be Binomial$(n,t_2)$. Some thought reveals that the latter is proven if we show that $X \leqst Y$, where $\leqst$ denotes the usual stochastic order. Indeed, by definition, $X \leqst Y$ if $\Pr(X > x) \leq \Pr(Y > x)$ for all $x \in \R$. According for instance to Theorem~1.C.1 in \cite{ShaSha07}, $X \leqst Y$ will hold if $X \leqlr Y$, where $\leqlr$ denotes the likelihood ratio order.

Let $f$ be the probability mass function of $X$ and let $g$ be the probability mass function of $Y$. Then, by definition,
\begin{align*}
  f(x) = {n \choose x} t_1^x (1-t_1)^{n-x}, \qquad  g(x) = {n \choose x} t_2^x (1-t_2)^{n-x}, \qquad x \in \{0,\dots,n\}.
\end{align*}
According for instance to Section~1.C.1 in \cite{ShaSha07}, to prove that $X \leqlr Y$, by definition, we need to show that the function $x \mapsto g(x)/f(x)$ is increasing on $\{0,\dots,n\}$. To prove the latter, it suffices to show that the function
$$
h(x) = \frac{t_2^x (1-t_2)^{n-x}}{t_1^x (1-t_1)^{n-x}}, \qquad x \in [0,n],
$$
is increasing on $[0,n]$. To do so, we shall prove that the derivative of $h$ is positive on $[0,n]$. Let $l(x)=\log \{h(x)\}$, $x \in [0,n]$. Then, for any $x \in [0,n]$,
\begin{align*}
\frac{\dd}{\dd x} \{ h(x) \} &= \frac{\dd}{\dd x} \{ e^{l(x)} \} = e^{l(x)} \frac{\dd}{\dd x} \{ l(x) \} = e^{l(x)}  \frac{\dd}{\dd x} [ \log \{ h(x) \} ] = e^{l(x)}  \frac{\dd}{\dd x} \left[ \log \left\{ \frac{t_{2}^x (1-t_{2})^{n-x}}{t_{1}^x (1-t_{1})^{n-x}} \right\} \right]\\
                  &= e^{l(x)} \frac{\dd}{\dd x} \left\{ x\log(t_2) + (n-x)\log(1-t_2) - x\log(t_1) - (n-x)\log(1-t_1)\right\} = e^{l(x)} \log \left\{\frac{t_2(1-t_1)}{t_1(1-t_2)} \right\} \geq 0,
\end{align*}
since $0 < t_1 \leq t_2 < 1$ implies that $t_2/t_1 \geq 1$ and that $(1-t_1)/(1 - t_2) \geq 1$.
\end{proof}

\begin{lem}
  \label{lem:Beta:increasing}
 For any $n \in \N$, $\rho \in (0,n)$ and $w \in (0,1)$, the function $t \mapsto \bar \beta_{n,t,\rho}(w)$ is increasing on $[0,1]$.
\end{lem}

\begin{proof}[\bf Proof]
  Fix $n \in \N$ and $\rho \in (0,n)$. We need to prove that, for any $w \in (0,1)$, $\bar \beta_{n,t_1,\rho}(w) \leq \bar \beta_{n,t_2,\rho}(w)$ whenever $0 \leq t_1 \leq t_2 \leq 1$. Since, for any $w \in (0,1)$, $\bar \beta_{n,0,\rho}(w) = 0$ and $\bar \beta_{n,1,\rho}(w) = 1$, it suffices to prove that, for any $w \in (0,1)$, $\bar \beta_{n,t_1,\rho}(w) \leq \bar \beta_{n,t_2,\rho}(w)$ whenever $0 < t_1 \leq t_2 < 1$. Fix $0 < t_1 \leq t_2 < 1$ and let $X$ be Beta$(t_1\tau,(1-t_1)\tau)$ and $Y$ be Beta$(t_2\tau,(1-t_2)\tau)$, where $\tau = (n-\rho)/\rho > 0$. As explained in the proof of Lemma~\ref{lem:Bin:increasing}, it then suffices to prove that $X \leqlr Y$.

Let $f$ be the density of $X$ and $g$ be the density of $Y$. By definition, we have that
\begin{align*}
  f(x) = \frac{\Gamma(\tau)}{\Gamma(t_1\tau)\Gamma \{ (1-t_1)\tau \}} x^{t_1\tau - 1} (1-x)^{(1-t_1)\tau - 1}, \qquad g(x) = \frac{\Gamma(\tau)}{\Gamma(t_2\tau)\Gamma \{ (1-t_2)\tau \} } x^{t_2\tau - 1} (1-x)^{(1-t_2)\tau - 1}, \qquad x \in (0,1).
\end{align*}
According for instance to Section~1.C.1 in \cite{ShaSha07}, to prove that $X \leqlr Y$,  we need to show that the function $x \mapsto g(x)/f(x)$ is increasing on $(0,1)$. Let $l(x)=\log \{g(x)/f(x)\}$, $x \in (0,1)$. Then, we have that, for any $x \in (0,1)$,
\begin{align*}
\frac{\dd}{\dd x} \left\{ g(x)/f(x) \right\} &= \frac{\dd}{\dd x} \left\{e^{l(x)} \right\} = e^{l(x)} \frac{\dd}{\dd x} \left\{ l(x) \right\} = e^{l(x)} \frac{\dd}{\dd x} \left[ \log \left\{ \frac{\frac{\Gamma(\tau)}{\Gamma(t_2\tau)\Gamma \{(1-t_2)\tau \}} x^{t_2\tau - 1} (1-x)^{(1-t_2)\tau - 1}}{\frac{\Gamma(\tau)}{\Gamma(t_1\tau)\Gamma\{(1-t_1)\tau\}} x^{t_1\tau - 1} (1-x)^{(1-t_1)\tau - 1}} \right\} \right]\\
                  &= e^{l(x)}  \frac{\dd}{\dd x} \left[ \log \left\{\frac{\Gamma(t_1\tau)\Gamma((1-t_1)\tau)}{\Gamma(t_2\tau)\Gamma((1-t_2)\tau)}\right\}+ \log \left\{  \frac{ x^{t_2\tau - 1} (1-x)^{(1-t_2)\tau - 1}}{ x^{t_1\tau - 1} (1-x)^{(1-t_1)\tau - 1}} \right\} \right]\\
                  &= e^{l(x)} \frac{\dd}{\dd x} \left\{ (t_2-t_1)\tau \log(x) - (t_2-t_1)\tau \log(1-x)\right\} = e^{l(x)} (t_2-t_1)\tau \left(\frac{1}{x} + \frac{1}{1-x} \right) \geq 0.
\end{align*}
\end{proof}

\begin{lem}
  \label{lem:BetaBin:increasing}
 For any $n \in \N$, $\rho \in (1,n)$ and $w \in [0,n)$, the function $t \mapsto \bar \Bc_{n,t,\rho}(w)$ is increasing on $[0,1]$.
\end{lem}

\begin{proof}[\bf Proof]
  Fix $n \in \N$ and $\rho \in (1,n)$ and let $\tau = (n-\rho)/(\rho - 1) > 0$.  We need to prove that, for any $w \in [0,n)$, $\bar \Bc_{n,t_1,\rho}(w) \leq \bar \Bc_{n,t_2,\rho}(w)$ whenever $0 \leq t_1 \leq t_2 \leq 1$. For any $w \in [0,n)$, $\bar \Bc_{n,0,\rho}(w) = 0$ and $\bar \Bc_{n,1,\rho}(w) = 1$ so that it suffices to prove that, for any $w \in [0,n)$, $\bar \Bc_{n,t_1,\rho}(w) \leq \bar \Bc_{n,t_2,\rho}(w)$ whenever $0 < t_1 \leq t_2 < 1$. Fix $0 < t_1 \leq t_2 < 1$ and let $X$ be Beta-Binomial($n,t_1 \tau, (1-t_1)\tau)$ and $Y$ be Beta-Binomial($n,t_2 \tau, (1-t_2)\tau)$. The latter is then proven if we show that $X \leqst Y$.

  It is well-known that the distributions of $X$ and $Y$ can be viewed as compound distributions: let $\Theta_1$ (resp.\ $\Theta_2$) be Beta$(t_1\tau,(1-t_1)\tau)$ (resp.\ Beta$(t_2\tau,(1-t_2)\tau)$) and let $Z_p$ be Binomial$(n,p)$; then, with some abuse of notation, $X$ (resp.\ $Y$) has the same distribution as $Z_{\Theta_1}$ (resp.\ $Z_{\Theta_2}$).  From the proof of Lemma~\ref{lem:Bin:increasing}, we have that $Z_{p_1} \leqst Z_{p_2}$ whenever $0 < p_1 \leq p_2 < 1$, while from the proof of Lemma~\ref{lem:Beta:increasing}, we have that $\Theta_1 \leqst \Theta_2$. Theorem 1.A.6 in \cite{ShaSha07} then implies that $X \leqst Y$, which completes the proof.
\end{proof}

\begin{proof}[\bf Proofs of Propositions~\ref{prop:Bin:cond},~\ref{prop:BetaBin:cond} and~\ref{prop:Beta:cond}]
Proposition~\ref{prop:Bin:cond} (resp.\ Proposition~\ref{prop:BetaBin:cond}, Proposition~\ref{prop:Beta:cond}) is an immediate consequence of Lemmas~\ref{lem:Bin:cont} and~\ref{lem:Bin:increasing} (resp.\ Lemmas~\ref{lem:BetaBin:cont} and~\ref{lem:BetaBin:increasing}, Lemmas~\ref{lem:Beta:cont} and~\ref{lem:Beta:increasing}).
\end{proof}


\section{Proofs of Theorem~\ref{thm:Cb:n:nu}}
\label{proof:thm:Cb:n:nu}

The proof of Theorem~\ref{thm:Cb:n:nu} is based on two lemmas which can be seen as extensions of similar results stated in Section~3 of \cite{SegSibTsu17}. The first lemma involves the following condition which is implied by Condition~\ref{cond:var:W}.

\begin{cond}
  \label{cond:var:W:weak}
  There exists a positive sequence $h_n \downarrow 0$ such that, for any $n \in \N$, $\bm x \in (\R^d)^n$, $\bm u \in [0,1]^d$ and $j \in \{1,\dots,d\}$, $\Var( W_{j,u_j}^{\bm x}) \leq h_n$.
\end{cond}

\begin{lem}
  \label{lem:stochastic}
  Let $\Xb_n$ be a process in $\ell^\infty(\Delta \times [0,1]^d)$ such that, for all $\bm u \in [0,1]^d$ and $s \in [0,1]$, $\Xb_n(s,s,\bm u) = 0$. Furthermore, assume that $\Xb_n \leadsto \Xb$ in $ \ell^\infty(\Delta \times [0,1]^d)$ where $\Xb$ has continuous trajectories almost surely. Then, under Condition~\ref{cond:var:W:weak},
  \begin{align*}
    \sup_{(s,t,\bm u) \in \Delta \times[0, 1]^d} \left| \int_{[0,1]^d}\Xb_n(s,t,\bm w)\dd \nu_{\bm u}^{\sss \Xc_{\ip{ns}+1:\ip{nt}}}(\bm w)-\Xb_n(s,t,\bm u) \right| &= o_\Pr(1).
  \end{align*}
\end{lem}

\begin{proof}[\bf Proof]
  Let $|\cdot|_\infty$ denote the maximum norm on $\R^d$. For any $\eps>0$,
\begin{align*}
\sup_{\substack{(s,t) \in \Delta \\ \bm u \in [0, 1]^d}} & \left| \int_{[0,1]^d} \Xb_n(s,t,\bm w) \dd \nu_{\bm u}^{\sss \Xc_{\ip{ns}+1:\ip{nt}}}(\bm w) - \Xb_n(s,t,\bm u) \right| = \sup_{\substack{(s,t) \in \Delta \\ \bm u \in [0, 1]^d}} \left| \int_{[0,1]^d} \{ \Xb_n(s,t,\bm w) - \Xb_n(s,t,\bm u) \} \dd \nu_{\bm u}^{\sss \Xc_{\ip{ns}+1:\ip{nt}}}(\bm w) \right| \\
  \leq& \sup_{\substack{(s,t) \in \Delta \\ \bm u \in [0, 1]^d}} \left| \int_{\{\bm w \in [0,1]^d: |\bm u - \bm w|_\infty  \leq \eps \}} \{ \Xb_n(s,t,\bm w) - \Xb_n(s,t,\bm u) \} \dd \nu_{\bm u}^{\sss \Xc_{\ip{ns}+1:\ip{nt}}}(\bm w) \right| \\
                                             &+ \sup_{\substack{(s,t) \in \Delta \\ \bm u \in [0, 1]^d}} \left| \int_{\{\bm w \in [0,1]^d: |\bm u - \bm w|_\infty > \eps \}} \{ \Xb_n(s,t,\bm w) - \Xb_n(s,t,\bm u) \} \dd \nu_{\bm u}^{\sss \Xc_{\ip{ns}+1:\ip{nt}}}(\bm w) \right| \\
  \leq& \sup_{\substack{(s,t,\bm u, \bm w) \in \Delta \times [0,1]^{2d}\\|\bm u - \bm w|_{\infty} \leq \eps}} \left|\Xb_n(s,t,\bm w)-\Xb_n(s,t,\bm u) \right| + \sup_{\substack{(s,t) \in \Delta  \\ (\bm u, \bm w) \times [0,1]^{2d}}} \Big[ \left| \Xb_n(s,t,\bm w)-\Xb_n(s,t,\bm u) \right| \\
  &  \qquad \times \nu_{\bm u}^{\sss \Xc_{\ip{ns}+1:\ip{nt}}}(\{\bm w \in [0,1]^d:|\bm u - \bm w|_{\infty} > \eps \}) \Big].
\end{align*}
Since $\Xb_n \leadsto \Xb$ in $ \ell^\infty(\Delta \times [0,1]^d)$ and $\Xb$ has continuous trajectories almost surely, it is stochastically equicontinuous. Hence, for any given $\eta > 0$, we can choose $\eps=\eps(\eta)>0$ sufficiently small such that
\begin{equation*}
\limsup_{n \rightarrow \infty} \Pr \left\{\sup_{\substack{(s,t,\bm u, \bm w) \in \Delta \times [0,1]^{2d}\\|\bm u - \bm w|_{\infty} \leq \eps}} \left|\Xb_n(s,t,\bm w) - \Xb_n(s,t,\bm u) \right| > \eta \right\} \leq \eta.
\end{equation*}
Some thought then reveals that, to complete the proof, it suffices to show that
\begin{align}
  \label{eq:tmp}
  \sup_{\substack{(s,t) \in \Delta \\ (\bm u, \bm w) \times [0,1]^{2d}}} \Big[ \left|\Xb_n(s,t,\bm w)-\Xb_n(s,t,\bm u) \right| \nu_{\bm u}^{\sss \Xc_{\ip{ns}+1:\ip{nt}}}(\{\bm w \in [0,1]^d:|\bm u - \bm w|_{\infty}  > \eps \}) \Big] = o_\Pr(1).
\end{align}
For any $\delta > 0$, the supremum on the left-hand side of~\eqref{eq:tmp} is smaller than
\begin{multline*}
 \sup_{\substack{(s,t,\bm u, \bm w) \in \Delta \times [0,1]^{2d} \\ t - s \leq \delta}} \left[ \left|\Xb_n(s,t,\bm w)-\Xb_n(s,t,\bm u) \right| \times \nu_{\bm u}^{\sss \Xc_{\ip{ns}+1:\ip{nt}}}(\{\bm w \in [0,1]^d:|\bm u - \bm w|_{\infty} > \eps \}) \right] \\
 + \sup_{\substack{(s,t,\bm u, \bm w) \in \Delta \times [0,1]^{2d} \\ t - s > \delta}} \left[ \left|\Xb_n(s,t,\bm w)-\Xb_n(s,t,\bm u) \right| \times \nu_{\bm u}^{\sss \Xc_{\ip{ns}+1:\ip{nt}}}(\{\bm w \in [0,1]^d:|\bm u - \bm w|_{\infty} > \eps \}) \right],
\end{multline*}
which is in turn smaller than 
\begin{align*}
2 \sup_{\substack{(s,t,\bm u) \in \Delta \times [0,1]^d \\ t - s \leq \delta}} \left|\Xb_n(s,t,\bm u) \right| + 2 \sup_{(s,t,\bm u) \in \Delta \times [0,1]^d} \left|\Xb_n(s,t,\bm u) \right|  \times \sup_{\substack{(s,t,\bm u) \in \Delta \times [0,1]^d \\ t - s > \delta}} \nu_{\bm u}^{\sss \Xc_{\ip{ns}+1:\ip{nt}}}(\{\bm w \in [0,1]^d:|\bm u - \bm w|_{\infty} > \eps \}).
\end{align*}
Using again the fact that $\Xb_n$ is stochastically equicontinuous and that it vanishes on the subset $\{(s,s,\bm u) : s \in [0,1], \bm u \in [0,1]^d\}$ of $\Delta \times [0,1]^d$, for any given $\eta > 0$, we can choose $\delta=\delta(\eta)>0$ sufficiently small such that 
\begin{align*}
  \limsup_{n \rightarrow \infty} \Pr \left\{\sup_{\substack{(s,t,\bm u) \in \Delta \times [0,1]^d \\ t-s \leq \delta}} \left|\Xb_n(s,t,\bm u) \right| > \eta \right\} = \limsup_{n \rightarrow \infty} \Pr \left\{\sup_{\substack{(s,t,\bm u) \in \Delta \times [0,1]^d \\ t-s \leq \delta}} \left|\Xb_n(s,t,\bm u) - \Xb_n(s,s,\bm u) \right| > \eta \right\} \leq \eta.
\end{align*}
The convergence in~\eqref{eq:tmp} will then hold if we additionally show that
\begin{equation}
  \label{eq:tmp2}
  \sup_{(s,t,\bm u) \in \Delta \times [0,1]^d} \left|\Xb_n(s,t,\bm u) \right| \times \sup_{\substack{(s,t,\bm u) \in \Delta \times [0,1]^d \\ t - s > \delta}} \nu_{\bm u}^{\sss \Xc_{\ip{ns}+1:\ip{nt}}}(\{\bm w \in [0,1]^d:|\bm u - \bm w|_{\infty} > \eps \}) = o_\Pr(1).
\end{equation}
From the weak convergence of $\Xb_n$, we have that $\sup_{(s,t,\bm u) \in \Delta \times [0,1]^{d}} |\Xb_n(s,t,\bm u) |=O_\Pr(1)$ whereas the second factor in the product on the left-hand side of~\eqref{eq:tmp2} can be shown to converge to zero almost surely. Indeed, proceeding along the lines of \cite[Section~3]{SegSibTsu17} using Chebyshev's inequality and Condition~\ref{cond:var:W:weak}, for almost any sequence $\bm X_1,\bm X_2,\dots$, conditionally on $\bm X_1,\bm X_2,\dots$, we obtain that
\begin{align*}
  \sup_{\substack{(s,t,\bm u) \in \Delta \times [0,1]^d \\ t - s > \delta}} &\nu_{\bm u}^{\sss \Xc_{\ip{ns}+1:\ip{nt}}}(\{\bm w \in [0,1]^d:|\bm u - \bm w|_{\infty} > \eps \}) = \sup_{\substack{(s,t,\bm u) \in \Delta \times [0,1]^d \\ t - s > \delta}} \Pr \left\{ \left| \bm W_{\bm u}^{\sss \Xc_{\ip{ns}+1:\ip{nt}}}  - \bm u) \right|_{\infty} > \eps \mid \Xc_{\ip{ns}+1:\ip{nt}} \right\} \\
                                                                            &= \sup_{\substack{(s,t,\bm u) \in \Delta \times [0,1]^d \\ t - s > \delta}} \Pr \left[ \bigcup_{j=1}^d \left\{ \left| W_{j,u_j}^{\sss \Xc_{\ip{ns}+1:\ip{nt}}} - u_j \right| >\eps \right\} \mid \Xc_{\ip{ns}+1:\ip{nt}} \right] \\
                                                                            &\leq \sup_{\substack{(s,t,\bm u) \in \Delta \times [0,1]^d \\ t - s > \delta}} \sum_{j=1}^{d}\Pr \left\{ \left| W_{j,u_j}^{\sss \Xc_{\ip{ns}+1:\ip{nt}}}- u_j \right| >\eps \mid \Xc_{\ip{ns}+1:\ip{nt}} \right\} \leq \sup_{\substack{(s,t,\bm u) \in \Delta \times [0,1]^d \\ t-s > \delta}} \sum_{j=1}^{d}\frac{\Var(W_{j,u_j}^{\sss \Xc_{\ip{ns}+1:\ip{nt}}} \mid \Xc_{\ip{ns}+1:\ip{nt}} )}{\eps^2} \\
                                                                            &\leq \frac{d}{\eps^2} \sup_{\substack{(s,t) \in \Delta \\ t-s > \delta}} h_{\ip{nt} - \ip{ns}} \leq \frac{d}{\eps^2} \sup_{\substack{(s,t) \in \Delta\\ t-s > \delta}} h_{\ip{n(t-s) - 1}} \leq \frac{d}{\eps^2} h_{\ip{n \delta - 1}} = o(1).
\end{align*}
It follows that the second factor in the product on the left-hand side of~\eqref{eq:tmp2} converges almost surely to zero, which completes the proof.
\end{proof}

\begin{lem}
\label{lem:bias}
Assume that Conditions~\ref{cond:pd} and~\ref{cond:var:W} hold. Then, almost surely,
\begin{align*}
  \sup_{(s,t,\bm u) \in \Delta \times[0, 1]^d} &\sqrt{n}\lambda_n(s,t) \left| \int_{[0,1]^d} C(\bm w) \dd \nu_{\bm u}^{\sss \Xc_{\ip{ns}+1:\ip{nt}}}(\bm w)-C(\bm u) \right| = o(1).
\end{align*}
\end{lem}

\begin{proof}[\bf Proof]
Let $a_n = \ip{n^{1/3}}$, $n \in \N$. With  probability 1, one has  that
\begin{multline*}
\sup_{(s,t,\bm u) \in \Delta \times[0, 1]^d} \sqrt{n}\lambda_n(s,t) \left| \int_{[0,1]^d} C(\bm w) \dd \nu_{\bm u}^{\sss \Xc_{\ip{ns}+1:\ip{nt}}}(\bm w)-C(\bm u) \right| \\ = \max_{1 \leq k \leq l \leq n} \left\{ \frac{l-k+1}{\sqrt{n}} \sup_{\bm u \in [0, 1]^d}  \left| \int_{[0,1]^d} C(\bm w) \dd \nu_{\bm u}^{\sss \Xc_{k:l}}(\bm w)-C(\bm u) \right|\right\} \leq \max(L_n,M_n),
\end{multline*}
where
\begin{align}
  \nonumber
  L_n &= \max_{1 \leq k \leq l \leq n \atop l-k \leq a_n} \left\{ \frac{l-k+1}{\sqrt{n}} \sup_{\bm u \in [0, 1]^d}  \left| \int_{[0,1]^d} C(\bm w) \dd \nu_{\bm u}^{\sss \Xc_{k:l}}(\bm w)-C(\bm u) \right|\right\},\\
  \label{eq:Mn}
  M_n &= \max_{1 \leq k \leq l \leq n \atop l-k > a_n} \left\{ \sqrt{l-k+1} \sup_{\bm u \in [0, 1]^d}  \left| \int_{[0,1]^d} C(\bm w) \dd \nu_{\bm u}^{\sss \Xc_{k:l}}(\bm w)-C(\bm u) \right|\right\}.
\end{align}
Since $0 \leq C \leq 1$, we have that $\sup_{\bm u \in [0, 1]^d}  \left| \int_{[0,1]^d} C(\bm w) \dd \nu_{\bm u}^{\sss \Xc_{k:l}}(\bm w)-C(\bm u) \right| \leq 1$ almost surely for all $1 \leq k \leq l \leq n$ and, therefore, with  probability 1, that
$$
L_n \leq \max_{1 \leq k \leq l \leq n \atop l-k \leq a_n} \frac{l-k+1}{\sqrt{n}}  \leq \frac{a_n + 1}{\sqrt{n}} = o(1).
$$
The aim of the remainder of this proof is to show that the term $M_n$ in~\eqref{eq:Mn} converges to zero almost surely as well. To do so, it suffices to show that, for almost any sequence $\bm X_1,\bm X_2,\dots$, conditionally on $\bm X_1,\bm X_2,\dots$, $M_n$ converges to zero. We thus reason conditionally on $\bm X_1,\bm X_2,\dots$ in the rest of this proof.

Let $\eta > 0$ and let us show that $M_n$ can be made smaller than $d\eta$ provided $n$ is large enough. The forthcoming arguments are very close to those used in the proof of Proposition~3.4 in \citep{SegSibTsu17}. We provide all the steps nonetheless for the sake of completeness.

Given $\bm u, \bm w \in [0,1]^d$, set $\bm w(t) = \bm u + t (\bm w - \bm u)$, $t \in [0,1]$. The function $G(t) = C\{\bm w(t) \}$, $t \in [0,1]$, is continuous on $[0,1]$ and, by Condition~\ref{cond:pd}, is continuously differentiable on $(0,1)$ with derivative
$$
G'(t) = \sum_{j=1}^d (w_j - u_j) \dot C_j \{ \bm w(t)\}, \qquad t \in (0,1).
$$
By the fundamental theorem of calculus, $G(1) - G(0) = \int_0^1 G'(t) \dd t$, that is,
$$
C(\bm w) - C(\bm u) = \sum_{j=1}^d (w_j - u_j)  \int_0^1 \dot C_j \{ \bm w(t)\} \dd t.
$$
Note that, under Condition~\ref{cond:pd} and with the adopted conventions, some thought reveals that the previous equality holds no matter how $\bm u$ and $\bm w$ are chosen in $[0,1]^d$. Using Fubini's theorem, we then obtain that, for any $\bm u \in [0,1]^d$ and $1 \leq k \leq l \leq n$,
$$
\int_{[0,1]^d} \{ C(\bm w) - C(\bm u) \} \dd \nu_{\bm u}^{\sss \Xc_{k:l}}(\bm w)=\sum_{j=1}^d\int_0^1\left\{ \int_{[0,1]^d}(w_j-u_j)\dot{C}_j \{ \bm w(t)\} \dd \nu_{\bm u}^{\sss \Xc_{k:l}}(\bm w)\right\} \dd t,
$$
which implies that
\begin{align*}
  M_n  &= \max_{1 \leq k \leq l \leq n \atop l-k > a_n} \left\{ \sqrt{l-k+1} \sup_{\bm u \in [0, 1]^d} \left| \sum_{j=1}^d\int_0^1\left\{ \int_{[0,1]^d}(w_j-u_j)\dot{C}_j \{ \bm w(t)\} \dd \nu_{\bm u}^{\sss \Xc_{k:l}}(\bm w)\right\} \dd t \right|\right\}\\
       & \leq \max_{1 \leq k \leq l \leq n \atop l-k > a_n} \left\{ \sqrt{l-k+1} \sum_{j=1}^d \sup_{\bm u \in [0, 1]^d} \left| \int_0^1\left\{ \int_{[0,1]^d}(w_j-u_j)\dot{C}_j \{ \bm w(t)\} \dd \nu_{\bm u}^{\sss \Xc_{k:l}}(\bm w)\right\} \dd t \right|\right\} \leq \sum_{j=1}^d I_{j,n},
\end{align*}
where, for any $j \in \{1,\dots,d\}$,
\begin{equation*}
I_{j,n} = \max_{1 \leq k \leq l \leq n \atop l-k > a_n} \left\{ \sqrt{l-k+1} \int_0^1\sup_{\bm u \in [0, 1]^d}\left|\int_{[0,1]^d}(w_j-u_j)\dot{C}_j \{ \bm w(t) \} \dd \nu_{\bm u}^{\sss \Xc_{k:l}}(\bm w)\right|\dd t \right\}.
\end{equation*}
Fix $j \in \{1,\dots,d\}$. We shall now show that, provided $n$ is large enough, $I_{j,n}$ is smaller than $\eta$. For any $\delta \in (0,1/2]$, we have that $I_{j,n} \leq J_{j,n,\delta} + K_{j,n,\delta}$, where
\begin{align}
  \nonumber
  J_{j,n,\delta} &= \max_{1 \leq k \leq l \leq n \atop l-k > a_n} \left\{ \sqrt{l-k+1} \int_0^1\sup_{\substack{\bm u \in [0, 1]^d \\ u_j \in [0,\delta)\cup(1-\delta,1]}} \left| \int_{[0,1]^d}(w_j-u_j) \dot{C}_j\{\bm w(t)\} \dd \nu_{\bm u}^{\sss \Xc_{k:l}}(\bm w) \right| \dd t \right\}, \\
  \label{eq:K:j:n:delta}
  K_{j,n,\delta} &= \max_{1 \leq k \leq l \leq n \atop l-k > a_n} \left\{ \sqrt{l-k+1} \int_0^1\sup_{\substack{\bm u \in [0, 1]^d \\ u_j \in [\delta,1-\delta]}} \left| \int_{[0,1]^d}(w_j-u_j) \dot{C}_j\{\bm w(t)\} \dd \nu_{\bm u}^{\sss \Xc_{k:l}}(\bm w) \right| \dd t \right\}.
\end{align}

\emph{Term $J_{j,n,\delta}$:} From the monotonicity and Lipschitz continuity of $C$, $0 \leq \dot C_j \leq 1$ \cite[see, e.g.,][Section 2.2]{Nel06} and therefore, using additionally Hölder's inequality and Condition~\ref{cond:var:W},
\begin{align*}
 J_{j,n,\delta} \leq& \max_{1 \leq k \leq l \leq n \atop l-k > a_n} \left\{ \sqrt{l-k+1} \sup_{\substack{\bm u \in [0, 1]^d \\ u_j \in [0,\delta)\cup(1-\delta,1]}}\int_{[0,1]^d}|w_j-u_j|\dd \nu_{\bm u}^{\sss \Xc_{k:l}}(\bm w)\right\} \\
  \leq& \max_{1 \leq k \leq l \leq n \atop l-k > a_n} \left\{ \sqrt{l-k+1} \sup_{\substack{\bm u \in [0, 1]^d \\ u_j \in [0,\delta)\cup(1-\delta,1]}} \sqrt{ \int_{[0,1]^d} \{ w_j-\Ex(W_{j,u_j}^{\sss \Xc_{k:l}} \mid \Xc_{k:l}) \}^2\dd \nu_{\bm u}^{\sss \Xc_{k:l}}(\bm w)}\right\} \\
  =& \max_{1 \leq k \leq l \leq n \atop l-k > a_n} \left\{ \sqrt{l-k+1} \sup_{\substack{\bm u \in [0, 1]^d \\ u_j \in [0,\delta)\cup(1-\delta,1]}} \sqrt{\Var(W_{j,u_j}^{\sss \Xc_{k:l}} \mid \Xc_{k:l})}\right\} \leq \sup_{u \in [0,\delta)\cup(1-\delta,1]}\sqrt{\kappa u(1-u)}.
\end{align*}
Hence, we can choose $\delta = \delta(\eta, \kappa)$ sufficiently small such that $J_{j,n,\delta} < \eta/3$.

\emph{Term $K_{j,n,\delta}$:} Since $\int_{[0,1]^d}(w_j - u_j)\dd \nu_{\bm u}^{\sss \Xc_{k:l}}(\bm w) = 0$ for all $1 \leq k \leq l \leq n$, we can rewrite the term $K_{j,n,\delta}$ in~\eqref{eq:K:j:n:delta} as
$$
 K_{j,n,\delta} = \max_{1 \leq k \leq l \leq n \atop l-k > a_n}\left\{\sqrt{l-k+1} \int_0^1\sup_{\substack{\bm u \in [0, 1]^d \\ u_j \in [\delta,1-\delta]}} \left| \int_{[0,1]^d}(w_j-u_j) \left[ \dot{C}_j\{\bm w(t)\} -\dot C_j(\bm u) \right] \dd \nu_{\bm u}^{\sss \Xc_{k:l}}(\bm w) \right| \dd t\right\}.
$$
Then, for any $\eps \in (0, \delta/2)$,
\begin{align*}
  K_{j,n,\delta} \leq& \max_{1 \leq k \leq l \leq n \atop l-k > a_n} \left\{\sqrt{l-k+1} \int_0^1\left[\sup_{\substack{\bm u \in [0, 1]^d \\ u_j \in [\delta,1-\delta]}} \left|\int_{\{\bm w \in [0,1]^d : |\bm w-\bm u|_{\infty}\leq \eps \}}(w_j-u_j)\left[ \dot{C}_j \{ \bm w(t) \} - \dot{C}_j(\bm u) \right]\dd \nu_{\bm u}^{\sss \Xc_{k:l}}(\bm w)\right|\right]\dd t \right\} \\
                     &+ \max_{1 \leq k \leq l \leq n \atop l-k > a_n} \left\{\sqrt{l-k+1} \int_0^1\left[\sup_{\substack{\bm u \in [0, 1]^d \\ u_j \in [\delta,1-\delta]}} \left|\int_{\{\bm w \in [0,1]^d:|\bm w-\bm u|_{\infty}> \eps\}}(w_j-u_j)\left[ \dot{C}_j \{ \bm w(t) \} - \dot{C}_j(\bm u) \right]\dd \nu_{\bm u}^{\sss \Xc_{k:l}}(\bm w)\right|\right]\dd t\right\} \\
   \leq& \; K_{j,n,\delta,\eps}' + K_{j,n,\delta,\eps}'',
\end{align*}
where
\begin{align*}
  K_{j,n,\delta,\eps}' &= \sup_{\substack{\bm u \in [0, 1]^d \\ u_j \in [\delta,1-\delta]}} \sup_{\substack{\bm w \in [0,1]^d \\ |\bm w-\bm u|_{\infty}\leq \eps }}\left|\dot{C}_j(\bm w) - \dot{C}_j(\bm u)\right| \times \max_{1 \leq k \leq l \leq n \atop l-k > a_n}\left\{\sqrt{l-k+1} \sup_{\substack{\bm u \in [0, 1]^d \\ u_j \in [\delta,1-\delta]}}\int_{[0,1]^d}\left| w_j-u_j\right|\dd \nu_{\bm u}^{\sss \Xc_{k:l}}(\bm w)\right\}, \\
K_{j,n,\delta,\eps}'' &= \max_{1 \leq k \leq l \leq n \atop l-k > a_n} \left\{\sqrt{l-k+1}\sup_{\substack{\bm u \in [0, 1]^d \\ u_j \in [\delta,1-\delta]}} \int_{[0,1]^d}\1\{|\bm w-\bm u|_\infty>\eps\} |w_j-u_j|\dd \nu_{\bm u}^{\sss \Xc_{k:l}}(\bm w)\right\},
\end{align*}
where we have used the fact that, for all $\bm u \in [0,1]^d$ and $t \in (0,1)$,
$$
\sup_{\substack{\bm w \in [0,1]^d \\ |\bm w-\bm u|_{\infty}\leq \eps}} \left|\dot{C}_j \{ \bm w(t) \} - \dot{C}_j(\bm u)\right| \leq \sup_{\substack{\bm w \in [0,1]^d \\ |\bm w-\bm u|_{\infty}\leq \eps}} \left|\dot{C}_j(\bm w) - \dot{C}_j(\bm u)\right|
$$
and, again, that $0 \leq \dot C_j \leq 1$. Using Condition~\ref{cond:var:W}, we obtain that, for any $1 \leq k \leq l \leq n$,
\begin{align}
 \label{ineq:1}
\sup_{\substack{\bm u \in [0, 1]^d \\ u_j \in [\delta,1-\delta]}}\int_{[0,1]^d}\left| w_j-u_j\right|\dd \nu_{\bm u}^{\sss \Xc_{k:l}}(\bm w) \leq \sup_{\substack{\bm u \in [0, 1]^d \\ u_j \in [\delta,1-\delta]}} \sqrt{\Var(W_{j,u_j}^{\sss \Xc_{k:l}} \mid \Xc_{k:l})} \leq \sup_{u \in [0,1]} \sqrt{\frac{\kappa u(1-u)}{l-k+1}} \leq \sqrt{\frac{\kappa}{l-k+1}}.
\end{align}
Since Condition~\ref{cond:pd} holds, by uniform continuity of $\dot C_j$ on the set $\{\bm u \in [0, 1]^d : u_j \in [\delta/2,1-\delta/2] \}$, we can choose $\eps = \eps(\delta,\eta,\kappa)>0$ sufficiently small such that
\begin{align}
 \label{ineq:2}
\sup_{\substack{\bm u \in [0, 1]^d \\ u_j \in [\delta,1-\delta]}} \sup_{\substack{\bm w \in [0,1]^d \\ |\bm w-\bm u|_{\infty}\leq \eps}} \left|\dot{C}_j(\bm w) - \dot{C}_j(\bm u)\right| \leq \frac{\eta}{3\sqrt{\kappa}}.
\end{align}
Combining \eqref{ineq:1} and \eqref{ineq:2}, we obtain that $K_{j,n,\delta,\eps}' \leq \eta/3$. As far as $K_{j,n,\delta,\eps}''$ is concerned, using the Cauchy-Schwarz inequality,
\begin{align}
  \nonumber
  K_{j,n,\delta,\eps}'' &\leq \max_{1 \leq k \leq l \leq n \atop l-k > a_n} \left\{\sqrt{l-k+1}\sup_{\substack{\bm u \in [0, 1]^d \\ u_j \in [\delta,1-\delta]}} \sqrt{\int_{[0,1]^d}\1\{|\bm w-\bm u|_\infty > \eps \}  \dd \nu_{\bm u}^{\sss \Xc_{k:l}}(\bm w) \int_{[0,1]^d} (w_j-u_j)^2 \dd \nu_{\bm u}^{\sss \Xc_{k:l}}(\bm w)}\right\}\\
  \label{ineq:K:2}
                        &\leq \max_{1 \leq k \leq l \leq n \atop l-k > a_n}\left\{ \sqrt{l-k+1}  \sup_{\substack{\bm u \in [0, 1]^d \\ u_j \in [\delta,1-\delta]}} \sqrt{\nu_{\bm u}^{\sss \Xc_{k:l}}(\{\bm w \in [0,1]^d : |\bm w-\bm u|_\infty>\eps \})} \sup_{\substack{\bm u \in [0, 1]^d \\ u_j \in [\delta,1-\delta]}} \sqrt{\Var(W_{j,u_j}^{\sss \Xc_{k:l}} \mid \Xc_{k:l})}\right\}.
\end{align}
Then, using Chebyshev's inequality and Condition~\ref{cond:var:W}, we obtain that, for any $\bm u \in [0, 1]^d$ such that $u_j \in [\delta,1-\delta]$ and for any $1 \leq k \leq l \leq n$,
\begin{align}
  \nonumber
  \nu_{\bm u}^{\sss \Xc_{k:l}}(\{\bm w \in [0,1]^d:|\bm u - \bm w|_{\infty} > \eps \}) &=  \Pr \left[ \bigcup_{j=1}^d \left\{ \left| W_{j,u_j}^{\sss \Xc_{k:l}} - u_j \right| >\eps \right\} \mid \Xc_{k:l} \right] \leq \sum_{j=1}^{d}\Pr \left\{ \left| W_{j,u_j}^{\sss \Xc_{k:l}} - u_j \right| >\eps \mid \Xc_{k:l} \right\} \\
  \label{ineq:K:3}
                                                                                       &\leq  \sum_{j=1}^{d}\frac{\Var(W_{j,u_j}^{\sss \Xc_{k:l}} \mid \Xc_{k:l})}{\eps^2} \leq \frac{\kappa d}{(l-k+1) \eps^2}.
\end{align}
Furthermore, using Condition~\ref{cond:var:W}, for any $1 \leq k \leq l \leq n$,
\begin{align}
 \label{ineq:K:4}
\sup_{\substack{\bm u \in [0, 1]^d \\ u_j \in [\delta,1-\delta]}} \sqrt{\Var(W_{j,u_j}^{\sss \Xc_{k:l}} \mid \Xc_{k:l})} \leq  \sup_{\substack{\bm u \in [0, 1]^d \\ u_j \in [\delta,1-\delta]}} \sqrt{\frac{\kappa u_j(1-u_j)}{l-k+1}} \leq \sqrt{\frac{\kappa}{l-k+1}}
\end{align}
Hence, from~\eqref{ineq:K:2}, \eqref{ineq:K:3} and \eqref{ineq:K:4}, we obtain that
$$
K_{j,n,\delta,\eps}'' \leq \max_{1 \leq k \leq l \leq n \atop l-k > a_n} \left\{\sqrt{l-k+1} \times \sqrt{\frac{\kappa d}{(l-k+1) \eps^2}} \times \sqrt{\frac{\kappa}{l-k+1}} \right\} \leq \frac{\kappa \sqrt{d}}{\eps} \times \max_{1 \leq k \leq l \leq n \atop l-k > a_n} \sqrt{\frac{1}{l-k+1}} \leq \frac{\kappa \sqrt{d}}{\eps} \times  \sqrt{\frac{1}{a_n+1}},
$$
which implies that, for $n$ sufficiently large, $K_{j,n,\delta,\eps}'' < \eta/3$. Thus, provided that $n$ is large enough, this successively implies that $K_{j,n,\delta} \leq 2\eta/3$, that $I_{j,n} \leq \eta$ and finally, since $M_n \leq \sum_{j=1}^d I_{j,n}$, that $M_n \leq d \eta$. The latter holds conditionally on $\bm X_1, \bm X_2, \dots$ for almost any sequence $\bm X_1, \bm X_2, \dots$, which completes the proof.
\end{proof}

\begin{proof}[\bf Proof of Theorem~\ref{thm:Cb:n:nu}]
From~\eqref{eq:C:kl:nu} and~\eqref{eq:Cb:n:nu}, for any $(s,t) \in \Delta$ and $\bm u \in [0,1]^d$, we have that
\begin{align}
  \nonumber
  \Cb_n^\nu(s,t,\bm u) =& \sqrt{n}\lambda_n(s,t) \left[ \int_{[0,1]^d}\left\{ C_{\ip{ns} +1 :\ip{nt}}(\bm w) - C(\bm w) + C(\bm w) \right\} \dd \nu_{\bm u}^{\sss \Xc_{\ip{ns}+1:\ip{nt}}}(\bm w)-C(\bm u) \right]\\
  \label{eq:decomp:nu:proc}
  =& \int_{[0,1]^d}\Cb_n(s,t,\bm w) \dd \nu_{\bm u}^{\sss \Xc_{\ip{ns}+1:\ip{nt}}}(\bm w) + \sqrt{n}\lambda_n(s,t) \left\{ \int_{[0,1]^d} C(\bm w) \dd \nu_{\bm u}^{\sss \Xc_{\ip{ns}+1:\ip{nt}}}(\bm w) - C(\bm u) \right\}.
\end{align}
Following \cite{SegSibTsu17}, the first (resp.\ second) term on the right of~\eqref{eq:decomp:nu:proc} can be called the stochastic (resp.\ bias) term in the decomposition of $\Cb_n^\nu(s,t,\bm u)$. Notice that both terms are equal to zero for any $(s,t) \in \Delta$ such that $\ip{ns} = \ip{nt}$. It then immediately follows from the triangular inequality that
\begin{multline*}
\sup_{(s,t) \in \Delta \atop \bm u \in [0,1]^d} \left|\Cb_n^\nu(s,t,\bm u) - \Cb_n(s,t,\bm u) \right| \leq \sup_{(s,t) \in \Delta \atop \bm u \in [0,1]^d} \left| \int_{[0,1]^d}\Cb_n(s,t,\bm w)\dd \nu_{\bm u}^{\sss \Xc_{\ip{ns}+1:\ip{nt}}}(\bm w) - \Cb_n(s,t,\bm u) \right| \\
+ \sup_{(s,t) \in \Delta \atop \bm u \in [0,1]^d} \sqrt{n}\lambda_n(s,t) \left| \int_{[0,1]^d} C(\bm w) \dd \nu_{\bm u}^{\sss \Xc_{\ip{ns}+1:\ip{nt}}}(\bm w)-C(\bm u)  \right|,
\end{multline*}
and the desired result is finally a direct consequence of Lemmas~\ref{lem:stochastic} and~\ref{lem:bias}.
\end{proof}

\section*{Acknowledgments}

The authors would like to thank an Associate Editor and two anonymous Referees for their very constructive comments on an earlier version of this manuscript as well as Mark Holmes for fruitful discussions.

\bibliographystyle{myjmva}
\bibliography{biblio}

\end{document}